\documentclass[letterpaper,final,twoside,leqno]{amsart}

\usepackage[normalem]{ulem}
\usepackage{mathtools}
\usepackage{mathrsfs}  
\usepackage[arrow,curve,matrix]{xy}

\usepackage{scalerel}
\usepackage{comment}
\usepackage{xspace}

\newif\ifdraft
\draftfalse

\newcommand\reduced{strongly snc\xspace}

\newcommand\bdot{}
\newcommand\Lm{\sW_{\!m}}
\newcommand\smn{\ensuremath{\mcS\mcm_{n,\nu}}}
\newcommand\stn{\ensuremath{\mcS\mct_{n,\nu}}}

\numberwithin{figure}{section}

\usepackage{amsmath,amsfonts,amsthm,mathrsfs}
\usepackage{amssymb}
\usepackage{mathrsfs}

\DeclareFontFamily{OMS}{rsfs}{\skewchar\font'60}
\DeclareFontShape{OMS}{rsfs}{m}{n}{<-5>rsfs5 <5-7>rsfs7 <7->rsfs10 }{}
\DeclareSymbolFont{rsfs}{OMS}{rsfs}{m}{n}
\DeclareSymbolFontAlphabet{\scr}{rsfs}

\usepackage[bookmarks]{hyperref}
\usepackage[usenames,dvipsnames]{xcolor}
\hypersetup{
    colorlinks=true,
    linkcolor={blue!50!black},
    citecolor={green!50!black},
    urlcolor={red!80!black}
}

\usepackage[alphabetic,initials]{amsrefs}

\usepackage{enumitem}

\usepackage{chngcntr}

\ifdraft
\usepackage[notcite,notref,color]{showkeys}

\definecolor{labelkey}{gray}{0.5}
\fi

\usepackage{tikz}

\usepackage{tikz-cd}
\tikzset{commutative diagrams/arrow style=math font}

\usetikzlibrary{matrix,arrows}
\newlength{\myarrowsize} 

\pgfarrowsdeclare{cmto}{cmto}{
	\pgfsetdash{}{0pt} 
	\pgfsetbeveljoin 
	\pgfsetroundcap 
	\setlength{\myarrowsize}{0.6pt}
	\addtolength{\myarrowsize}{.5\pgflinewidth}
	\pgfarrowsleftextend{-4\myarrowsize-.5\pgflinewidth} 
	\pgfarrowsrightextend{.8\pgflinewidth}
}{
	\setlength{\myarrowsize}{0.6pt} 
  	\addtolength{\myarrowsize}{.5\pgflinewidth}  
	\pgfsetlinewidth{0.5\pgflinewidth}
	\pgfsetroundjoin
	\pgfpathmoveto{\pgfpoint{1.5\pgflinewidth}{0}}
	\pgfpatharc{-109}{-170}{4\myarrowsize}
	\pgfpatharc{10}{189}{0.58\pgflinewidth and 0.2\pgflinewidth}
	\pgfpatharc{-170}{-115}{4\myarrowsize+\pgflinewidth}
	\pgfpathclose
	\pgfusepathqfillstroke
	\pgfpathmoveto{\pgfpoint{1.5\pgflinewidth}{0}}
	\pgfpatharc{109}{170}{4\myarrowsize}
	\pgfpatharc{-10}{-189}{0.58\pgflinewidth and 0.2\pgflinewidth}
	\pgfpatharc{170}{115}{4\myarrowsize+\pgflinewidth}
	\pgfpathclose
	\pgfusepathqfillstroke
	\pgfsetlinewidth{2\pgflinewidth}
}

\pgfarrowsdeclare{cmonto}{cmonto}{
	\pgfsetdash{}{0pt} 
	\pgfsetbeveljoin 
	\pgfsetroundcap 
	\setlength{\myarrowsize}{0.6pt}
	\addtolength{\myarrowsize}{.5\pgflinewidth}
	\pgfarrowsleftextend{-4\myarrowsize-.5\pgflinewidth} 
	\pgfarrowsrightextend{.8\pgflinewidth}
}{
	\setlength{\myarrowsize}{0.6pt} 
  	\addtolength{\myarrowsize}{.5\pgflinewidth}  
	\pgfsetlinewidth{0.5\pgflinewidth}
	\pgfsetroundjoin
	\pgfpathmoveto{\pgfpoint{1.5\pgflinewidth}{0}}
	\pgfpatharc{-109}{-170}{4\myarrowsize}
	\pgfpatharc{10}{189}{0.58\pgflinewidth and 0.2\pgflinewidth}
	\pgfpatharc{-170}{-115}{4\myarrowsize+\pgflinewidth}
	\pgfpathclose
	\pgfusepathqfillstroke
	\pgfpathmoveto{\pgfpoint{1.5\pgflinewidth}{0}}
	\pgfpatharc{109}{170}{4\myarrowsize}
	\pgfpatharc{-10}{-189}{0.58\pgflinewidth and 0.2\pgflinewidth}
	\pgfpatharc{170}{115}{4\myarrowsize+\pgflinewidth}
	\pgfpathclose
	\pgfusepathqfillstroke
	\pgfpathmoveto{\pgfpoint{1.5\pgflinewidth-0.3em}{0}}
	\pgfpatharc{-109}{-170}{4\myarrowsize}
	\pgfpatharc{10}{189}{0.58\pgflinewidth and 0.2\pgflinewidth}
	\pgfpatharc{-170}{-115}{4\myarrowsize+\pgflinewidth}
	\pgfpathclose
	\pgfusepathqfillstroke
	\pgfpathmoveto{\pgfpoint{1.5\pgflinewidth-0.3em}{0}}
	\pgfpatharc{109}{170}{4\myarrowsize}
	\pgfpatharc{-10}{-189}{0.58\pgflinewidth and 0.2\pgflinewidth}
	\pgfpatharc{170}{115}{4\myarrowsize+\pgflinewidth}
	\pgfpathclose
	\pgfusepathqfillstroke
	\pgfsetlinewidth{2\pgflinewidth}
}

\pgfarrowsdeclare{cmhook}{cmhook}{
	\pgfsetdash{}{0pt} 
	\pgfsetbeveljoin 
	\pgfsetroundcap 
	\setlength{\myarrowsize}{0.6pt}
	\addtolength{\myarrowsize}{.5\pgflinewidth}
	\pgfarrowsleftextend{-4\myarrowsize-.5\pgflinewidth} 
	\pgfarrowsrightextend{.8\pgflinewidth}
}{
	\setlength{\myarrowsize}{0.6pt} 
  	\addtolength{\myarrowsize}{.5\pgflinewidth}  
 	\pgfsetdash{}{0pt}
	\pgfsetroundcap
	\pgfpathmoveto{\pgfqpoint{0pt}{-4.667\pgflinewidth}}
	\pgfpathcurveto
    {\pgfqpoint{4\pgflinewidth}{-4.667\pgflinewidth}}
    {\pgfqpoint{4\pgflinewidth}{0pt}}
    {\pgfpointorigin}
	\pgfusepathqstroke
}

\newenvironment{diagram*}[2]{%
\[%
\begin{tikzpicture}[>=cmto,baseline=(current bounding box.center),%
	to/.style={->,font=\scriptsize,cap=round},%
	into/.style={cmhook->,font=\scriptsize,cap=round},%
	onto/.style={-cmonto,font=\scriptsize,cap=round},%
	math/.style={matrix of math nodes, row sep=#2, column sep=#1,%
		text height=1.5ex, text depth=0.25ex}]%
}{%
\end{tikzpicture}%
\]%
\ignorespacesafterend%
}

 \DeclareMathOperator{\coker}{coker}
 \DeclareMathOperator{\Sym}{Sym}
 
\DeclareMathOperator{\disc}{disc}
\DeclareMathOperator{\pr}{pr}
\DeclareMathOperator{\R}{\myR}

\newcommand{\wtilde}{\widetilde}

\newcommand{\hooklongrightarrow}{\lhook\joinrel\longrightarrow}

\newcommand{\theoremref}[1]{\hyperref[#1]{Theorem~\ref*{#1}}}
\newcommand{\lemmaref}[1]{\hyperref[#1]{Lemma~\ref*{#1}}}
\newcommand{\definitionref}[1]{\hyperref[#1]{Definition~\ref*{#1}}}
\newcommand{\propositionref}[1]{\hyperref[#1]{Proposition~\ref*{#1}}}
\newcommand{\conjectureref}[1]{\hyperref[#1]{Conjecture~\ref*{#1}}}
\newcommand{\corollaryref}[1]{\hyperref[#1]{Corollary~\ref*{#1}}}
\newcommand{\exampleref}[1]{\hyperref[#1]{Example~\ref*{#1}}}

\makeatletter
\let\old@caption\caption
\renewcommand*{\caption}[1]{%
  \setcounter{figure}{\value{equation}}%
  \stepcounter{equation}%
  \old@caption{#1}\relax%
}
\makeatother

\newcounter{intro}

\newtheorem{intro-conjecture}[intro]{Conjecture}
\newtheorem{intro-corollary}[intro]{Corollary}
\newtheorem{intro-theorem}[intro]{Theorem}

\newcommand{\parref}[1]{\hyperref[#1]{\S\ref*{#1}}}

\makeatletter
\newcommand*\if@single[3]{%
  \setbox0\hbox{${\mathaccent"0362{#1}}^H$}%
  \setbox2\hbox{${\mathaccent"0362{\kern0pt#1}}^H$}%
  \ifdim\ht0=\ht2 #3\else #2\fi
  }
\newcommand*\rel@kern[1]{\kern#1\dimexpr\macc@kerna}
\newcommand*\widebar[1]{\@ifnextchar^{{\wide@bar{#1}{0}}}{\wide@bar{#1}{1}}}
\newcommand*\wide@bar[2]{\if@single{#1}{\wide@bar@{#1}{#2}{1}}{\wide@bar@{#1}{#2}{2}}}
\newcommand*\wide@bar@[3]{%
  \begingroup
  \def\mathaccent##1##2{%
    \if#32 \let\macc@nucleus\first@char \fi
    \setbox\z@\hbox{$\macc@style{\macc@nucleus}_{}$}%
    \setbox\tw@\hbox{$\macc@style{\macc@nucleus}{}_{}$}%
    \dimen@\wd\tw@
    \advance\dimen@-\wd\z@
    \divide\dimen@ 3
    \@tempdima\wd\tw@
    \advance\@tempdima-\scriptspace
    \divide\@tempdima 10
    \advance\dimen@-\@tempdima
    \ifdim\dimen@>\z@ \dimen@0pt\fi
    \rel@kern{0.6}\kern-\dimen@
    \if#31
      \overline{\rel@kern{-0.6}\kern\dimen@\macc@nucleus\rel@kern{0.4}\kern\dimen@}%
      \advance\dimen@0.4\dimexpr\macc@kerna
      \let\final@kern#2%
      \ifdim\dimen@<\z@ \let\final@kern1\fi
      \if\final@kern1 \kern-\dimen@\fi
    \else
      \overline{\rel@kern{-0.6}\kern\dimen@#1}%
    \fi
  }%
  \macc@depth\@ne
  \let\math@bgroup\@empty \let\math@egroup\macc@set@skewchar
  \mathsurround\z@ \frozen@everymath{\mathgroup\macc@group\relax}%
  \macc@set@skewchar\relax
  \let\mathaccentV\macc@nested@a
  \if#31
    \macc@nested@a\relax111{#1}%
  \else
    \def\gobble@till@marker##1\endmarker{}%
    \futurelet\first@char\gobble@till@marker#1\endmarker
    \ifcat\noexpand\first@char A\else
      \def\first@char{}%
    \fi
    \macc@nested@a\relax111{\first@char}%
  \fi
  \endgroup
}
\makeatother

\usepackage{amsbsy}

\usepackage{marvosym}

\DeclareMathAlphabet{\smallchanc}{OT1}{pzc}%
                                 {m}{it}
\DeclareFontFamily{OT1}{pzc}{}
\DeclareFontShape{OT1}{pzc}{m}{it}%
             {<-> s * [1.100] pzcmi7t}{}
\DeclareMathAlphabet{\mathchanc}{OT1}{pzc}%
                                 {m}{it}

\newcommand{\mcR}{\mathchanc{R}}
\newcommand{\mcS}{\mathchanc{S}}

\newcommand{\mcm}{\mathchanc{m}}

\newcommand{\mct}{\mathchanc{t}}

\newcommand{\sA}{\mathscr{A}}

\newcommand{\sE}{\mathscr{E}}
\newcommand{\sF}{\mathscr{F}}
\newcommand{\sG}{\mathscr{G}}

\newcommand{\sK}{\mathscr{K}}
\newcommand{\sL}{\mathscr{L}}
\newcommand{\sM}{\mathscr{M}}
\newcommand{\sN}{\mathscr{N}}
\newcommand{\sO}{\mathscr{O}}

\newcommand{\sW}{\mathscr{W}}
\newcommand{\sX}{\mathscr{X}}

\newcommand{\sff}{{\sf f}}

\newcommand{\bC}{\mathbb{C}}

\newcommand{\bN}{\mathbb{N}}

\newcommand{\bP}{\mathbb{P}}
\newcommand{\bQ}{\mathbb{Q}}

\newcommand{\bZ}{\mathbb{Z}}

\newcommand{\cH}{\mathcal{H}}

\DeclareSymbolFont{largesymbolsA}{U}{jkpexa}{m}{n}
\SetSymbolFont{largesymbolsA}{bold}{U}{jkpexa}{bx}{n}
\DeclareMathSymbol{\varprod}{\mathop}{largesymbolsA}{16}
\makeatletter
\newcommand{\LeftEqNo}{\let\veqno\@@leqno}
\makeatother
\newcommand\exist{\exists\,}

\newcommand{\into}{\hookrightarrow}
\newcommand{\longinto}{\lhook\joinrel\relbar\joinrel\longrightarrow}

\newcommand{\properideal}%
        {\subsetneq}

\newcommand{\wt}{\widetilde}

\newcommand{\rup}[1]{\lceil{#1}\rceil}

\newcommand{\leteq}{\colon\!\!\!=}

\newcommand\dash[1]{\rule[-.2ex]{#1}{.4pt}}

\DeclareMathOperator{\codim}{codim}

\newcommand{\sHom}[0]{{\mathchanc{Hom}}}

\DeclareMathOperator{\id}{{id}}

\DeclareMathOperator{\rank}{{rank}}

\DeclareMathOperator{\red}{red}

\DeclareMathOperator{\reg}{reg}

\DeclareMathOperator{\supp}{{supp}}
\DeclareMathOperator{\sym}{{Sym}}

\DeclareMathOperator{\skvert}{{\,\vert\,}}

\newcommand{\factor}[2]{\left. \raise .2em\hbox{\ensuremath{#1}\vphantom{$I^d$}}
\hskip -.1em \right/ \hskip -.4em \raise -.3em\hbox{\ensuremath{#2}}}%
\newcommand\mtimes[3]{{\varprod_{#1}^{#2}}_{\raise 1ex \hbox{\scriptsize #3}}}%

\newcommand{\myR}{{\mcR\!}}

\newcommand{\sblank}{\dash{.6em}}

\def\dimcoh#1.#2.#3.{h^{#1}(#2,#3)}
\def\hypcoh#1.#2.#3.{\mathbb H_{\vphantom{l}}^{#1}(#2,#3)}
\def\loccoh#1.#2.#3.#4.{H^{#1}_{#2}(#3,#4)}
\def\dimloccoh#1.#2.#3.#4.{h^{#1}_{#2}(#3,#4)}
\def\lochypcoh#1.#2.#3.#4.{\mathbb H^{#1}_{#2}(#3,#4)}
\def\seslong#1.#2.#3.{0  \longrightarrow  #1   \longrightarrow 
 #2 \longrightarrow #3 \longrightarrow 0} 
\def\sesshort#1.#2.#3.{0
 \rightarrow #1 \rightarrow #2 \rightarrow #3 \rightarrow 0}
\def\Iff#1#2#3{
\hfil\hbox{\hsize =#1
\vtop{\noin #2}
\hskip.5cm 
\lower.5\baselineskip\hbox{$\Leftrightarrow$}\hskip.5cm
\vtop{\noin #3}}\hfil\medskip}
\newcommand{\union}\cup
\newcommand{\intersect}\cap
\newcommand{\Union}\bigcup
\newcommand{\Intersect}\bigcap
\def\myoplus#1.#2.{\underset #1 \to {\overset #2 \to \oplus}}

\newcommand{\resto}[1]{\raise -.5ex\hbox{$\vert$}_{#1}}

\newcommand{\ses}{short exact sequence\xspace}

\newcommand\noin{\noindent}

   [2017/06/13 v0.1 finetuning mabliautoref and theorem setup]

\RequirePackage{hyperref}

\newcommand{\sectionsize}{} 
\newcommand{\theoremsize}{} 

\renewcommand{\subsectionautorefname}{\sectionsize\sf \subsectionautorefname}

\makeatletter
\@ifdefinable\equationname{\let\equationname\equationautorefname}
\def\equationautorefname~#1\@empty\@empty\null{\protect{\theoremsize
    (#1\@empty\@empty\null)}}%
\@ifdefinable\AMSname{\let\AMSname\AMSautorefname}
\def\AMSautorefname~#1\@empty\@empty\null{
  ( #1\@empty\@empty\null)}%
\@ifdefinable\itemname{\let\itemname\itemautorefname}
\def\itemautorefname~#1\@empty\@empty\null{\theoremsize{%
    {#1}}\@empty\@empty\null%
}%
\makeatother

\RequirePackage{amsthm}
\RequirePackage{aliascnt}
\newcommand{\basetheorem}[3]{%
    \newtheorem{#1}{#2}[#3]
    \newtheorem*{#1*}{#2}
    \expandafter\def\csname #1autorefname\endcsname{#2}
}%
\newcommand{\maketheorem}[3]{%
    \newaliascnt{#1}{#2}
    \newtheorem{#1}[#1]{\theoremsize #3}
    \aliascntresetthe{#1}
    \expandafter\def\csname #1autorefname\endcsname{#3}
    \newtheorem{#1*}{#3}
}%
\newcommand{\baseremark}[3]{%
    \newtheorem{#1}{#2}{#3}
    \newtheorem*{#1*}{#2}
    \expandafter\def\csname #1autorefname\endcsname{#2}
}%
\newcommand{\makeremark}[3]{%
    \newaliascnt{#1}{#2}
    \newtheorem{#1}[equation]{#3}
    \aliascntresetthe{#1}
    \expandafter\def\csname #1autorefname\endcsname{\theoremsize\sf #3}
    \newtheorem{#1*}{#3}
}%
\theoremstyle{plain}   

\basetheorem{theorem}{Theorem}{section}

\theoremstyle{definition}    

\newcommand\aug{twisted\xspace}
\newcommand\Aug{Twisted\xspace}
\newcommand{\tma}[1]{t_{m,#1}}

\newcommand{\hide}[1]{}

\setlist[enumerate, 1]{font=\upshape}

\numberwithin{figure}{section}

\DeclareMathOperator{\Gal}{Gal}
\DeclareMathOperator{\Spec}{Spec}
\DeclareMathOperator{\vol}{vol}

\theoremstyle{plain}
\maketheorem{prop}{theorem}{Proposition}
\maketheorem{cor}{theorem}{Corollary}
\maketheorem{lem}{theorem}{Lemma}
\maketheorem{conj}{theorem}{Conjecture}

\newtheorem{proposition}[prop]{Proposition}
\newtheorem{corollary}[cor]{Corollary}
\newtheorem{lemma}[lem]{Lemma}

\theoremstyle{definition}
\maketheorem{defini}{theorem}{Definition}
\newtheorem{definition}[defini]{Definition}

\theoremstyle{remark}
\maketheorem{rem}{theorem}{Remark}
\newtheorem{remark}[rem]{Remark}
\maketheorem{notat}{theorem}{Notation}
\maketheorem{notr}{theorem}{Notation-Remark}
\maketheorem{defnot}{theorem}{Definition-Notation}
\newtheorem{notation}[notat]{Notation}
\newtheorem{not-rem}[notr]{Notation-Remark}
\newtheorem{def-not}[defnot]{Definition-Notation}

\maketheorem{set}{theorem}{Set-up}
\newtheorem{set-up}[set]{Set-up}

\maketheorem{term}{theorem}{Terminology}

\maketheorem{fac}{theorem}{Fact}
\newtheorem{fact}[fac]{Fact}

\maketheorem{ass}{theorem}{Assumption}

\maketheorem{conn}{theorem}{Convention}

\maketheorem{clai}{theorem}{Claim}
\newtheorem{claim}[clai]{Claim}

\maketheorem{sclai}{theorem}{Subclaim}

\setlist[enumerate]{label=(\thetheorem.\arabic*), before={\setcounter{enumi}{\value{equation}}}, after={\setcounter{equation}{\value{enumi}}}}

\numberwithin{equation}{theorem}

\makeatletter
\let\amsmath@bigm\bigm

\renewcommand{\bigm}[1]{%
  \ifcsname fenced@\string#1\endcsname
    \expandafter\@firstoftwo
  \else
    \expandafter\@secondoftwo
  \fi
  {\expandafter\amsmath@bigm\csname fenced@\string#1\endcsname}%
  {\amsmath@bigm#1}%
}
\newcommand{\DeclareFence}[2]{\@namedef{fenced@\string#1}{#2}}
\makeatother

\DeclareFence{\mid}{|}

\begin{document}

\title[Arakelov inequalities in higher dimensions]{Arakelov inequalities in higher
  dimensions}

\author{S\'andor J Kov\'acs} \address{S\'andor Kov\'acs, Department of Mathematics,
  University of Washington, Box 354350, Seattle, Washington, 98195, U.S.A}
\email{\href{mailto:skovacs@uw.edu }{skovacs@uw.edu }}
\urladdr{\href{http://sites.math.washington.edu/~kovacs/current/}
  {http://sites.math.washington.edu/~kovacs/current/}}

\author{Behrouz Taji} \address{Behrouz Taji, School of Mathematics and Statistics,
The University of New South Wales Sydney, NSW 2052 Australia}
\email{\href{mailto:b.taji@unsw.edu.au}{b.taji@unsw.edu.au}}
\urladdr{\href{https://web.maths.unsw.edu.au/~btaji//}
  {https://web.maths.unsw.edu.au/~btaji/}}

\thanks{S\'andor Kov\'acs was supported in part by NSF Grants DMS-1565352,
  DMS-1951376 and DMS-2100389.}




\setlength{\parskip}{0.19\baselineskip}


\begin{abstract}
  We develop a Hodge theoretic invariant for families of projective manifolds that
  measures the potential failure of an Arakelov-type inequality in higher dimensions,
  one that naturally generalizes the classical Arakelov inequality over regular
  quasi-projective curves.  We show that for families of manifolds with ample
  canonical bundle this invariant is uniformly bounded.  As a consequence we
  establish that such families over a base of arbitrary dimension satisfy the
  aforementioned Arakelov inequality, answering a question of Viehweg.
\end{abstract}

\maketitle

\tableofcontents

\newcommand\hmarginpar[1]{}

\section{Introduction}

While numerical invariants play a central role in classification in all fields of
mathematics,
it is often
very difficult to compute their exact value. As a result we opt for the next best
thing: try to give estimates by finding upper or lower bounds. In algebraic geometry, and
in particular in the construction of moduli spaces, giving bounds for certain
invariants 
provides a fundamnatal tool. 
Without such bounds it would
be extremely difficult to find reasonably-behaved moduli spaces; for example, we could not even hope for
such spaces to be of finite type.

One of the early examples of such bounds, with an eye towards the construction of moduli
spaces of higher dimensional varieties, is Matsusaka's Big Theorem \cite{Mat72}. 
Boundedness questions are present in many
other more or less related questions, such as Mordell's Conjecture, Lang's
Conjecture, or Shafarevich's Conjecture. The latter, and its more modern
generalizations, are the most relevant to the present work.

Shafarevich \cite{Shaf63} conjectured that there are only finitely many
non-isotrivial families of smooth projective curves of fixed genus ($\geq 2$) over a fixed
curve. Parshin \cite{Parshin68} and Arakelov \cite{Arakelov71} proved this conjecture
in two steps: \emph{boundedness}, that is, there are only finitely many
deformation types of such families, and \emph{rigidity}; those families
are actually rigid, so each one is the \emph{only one} in its deformation type.


Boundedness can be roughly translated to some associated parameter scheme being of
finite type. These parameter spaces are often constructed via an appropriate Hilbert
scheme and hence being of finite type is closely related to bounding the degree of an
ample line bundle. In fact, already Arakelov used this idea to prove boundedness in
order to prove Shafarevich's conjecture in the curve case.

More generally, we consider a smooth projective family of
canonically polarized varieties $\pi: U \to V$. Then $V$ maps to a moduli space parametrizing the
fibers. This target moduli space is equipped with an ample line bundle
cf.~\cite{Kollar90,Fuj18,KP17}. The pullback of this line bundle to $V$ is
$\det \pi_* \omega^m_{U/V}$ (for some well-chosen $m>0$ and up to a suitable
power). Therefore, in order to carry out the above sketched plan for the boundedness
problem, one would need to uniformly bound the degree of this line bundle.

This is exactly what Arakelov did. He established such a universal bound for all
families of curves of genus at least $2$ over base spaces of dimension one
\cite{Arakelov71}. More precisely, he showed that, for every sufficiently large
$m\in \bN$, there is a polynomial function $b_{m,g} \in \bZ_{>0}[x_1, x_2]$,
depending only on $m$ and a fixed integer $g\in\bN$, $g\geq 2$, 
such that the inequality
\begin{equation}
  \label{ARAK}
  \deg (\det f_* \omega^m_{X/B})  \leq  b_{m,g} \big(  g(B), \deg(D) \big) \tag{$\star$}
\end{equation}
\hmarginpar{\tiny \color{red} \bf S: It seems to me that at this point there is no
  need for a Hilbert polynomial. Fixing the genus of a curve fixes its Hilbert poly,
  so it seems that that's enough. So, I commented out the references to the Hilb poly
  and changed ``h'' to ``g'' \\
  \color{blue} B: Absolutely! Thanks! }%
holds for any smooth compactification $f: X\to B$ of \emph{any} non-isotrivial smooth
projective family $f_U: U\to V$ of curves of genus $g$ 
over a one dimensional base $V$, where $D:= B\setminus V$.  In fact Arakelov showed
that the coefficients of $b_{m,g}$ are themselves purely $g$-dependent functions of
$m$ and $r_m: = \rank (f_*\omega^m_{X/B})$.
  
This result was partially generalized to the case of higher dimensional fibers in
\cites{Kovacs96e,Kovacs97a,Kovacs00a}.  Subsequently, Bedulev and Viehweg
\cite{Bedulev-Viehweg00} proved a further generalization of Arakelov's inequality for
families of canonically polarized manifolds,
\emph{still} over curves.  Other, more Hodge theoretic analogues of \eqref{ARAK} were
also established by Deligne \cite{Deligne87} and Peters \cite{Pet00} (see
Subsection~\ref{subsect:related} below or \cite{Vie08} for a more detailed account).

The equation \eqref{ARAK} became known as Arakelov's inequality. To see its
usefulness the reader is invited to consult \cite{Vie08} for a survey of related
results available at the time and \S8 of that paper for several open questions. Based
on \cite{KoL10}, Viehweg and others speculated that the
inequality \eqref{ARAK} should have analogues over higher dimensional base spaces. In
fact, at the end of his survey \cite{Vie08} Viehweg explains how a higher dimensional Arakelov
inequality would be useful, and goes on to say that none of the known methods (at the
time) give any hope of obtaining it \cite[\S 8:III,IV]{Vie08}.

\autoref{def:ARAK} provides a natural higher dimensional analogue of \eqref{ARAK} and
the main result of the present paper is that under natural assumptions this
inequality holds for canonically polarized families.


\begin{remark}\label{rem:volume-instead-of-Hilb-poly}
  It used to be customary to fix a Hilbert polynomial when one is discussing moduli
  functors in order to have a finite type moduli space. Recently Koll\'ar showed that
  for families of stable varieties it is actually enough to fix the volume of the
  canonical divisor (which appears as a coefficient of the canonical Hilbert
  polynomial) \cite{ModBook}*{5.1,6.19}. As this is now the standard, we will follow
  this approach and refer to the volume of the canonical divisor as the
  \emph{canonical volume}. More generally, we will follow the terminology of the
  \cite{ModBook} on everything related to moduli spaces of stable varieties. In order
  to keep the introduction manageable, we only address some of the details on moduli
  spaces in \autoref{sect:Section4-invariant}.
\end{remark}

\begin{definition}[Higher dimensional Arakelov type inequalities]
  \label{def:ARAK}
  Let $V$ be a smooth quasi-projective variety of dimension $d$ and $B$ a smooth
  compactification of $V$ such that $B\setminus D \simeq V$, with $D$ being a reduced divisor on 
  $B$ having  
  simple normal crossing support.  Further
  let $f_U:U\to V$ be a smooth family of projective varieties and let $X$ be a smooth
  compactification of $U$ such that there exists a projective morphism $f:X\to B$
  with $f\resto U=f_U$.  We will refer to these by saying that (the pair)
  \emph{$(B,D)$ is a smooth compactification of $V$} and that $f:X\to B$ is a
  \emph{smooth compactification of $f_U:U\to V$}.

  Still working with the above notations, let $H$ be an ample Cartier divisor on
  $B$ and set $\smn(V)$ to denote the class of smooth projective families, $f_U:U\to V$,
  of canonically polarized varieties of relative dimension $n$ and canonical volume
  $\nu=\vol(K_{U_t})\leteq K_{U_t}^n$ over $V$.  Members of a subclass of
  $\mcS \subseteq \smn(V)$ will be said to satisfy an \emph{Arakelov inequality}, if
  for all sufficiently large and divisible $m\in \bN$, there exists a function
  $b_{m,n,\nu}\in \bZ_{>0}[x_1, x_2]$, depending only on $m$, $n$, and $\nu$, for
  which the inequality
  \begin{equation}\label{ARAK2}
    \deg_H \big(  \det f_*\omega^m_{X/B}  \big)    \leq   b_{m,n,\nu} \left(  \deg_H (K_B+D)
      , \deg_H(D) \right)   
  \end{equation}
  holds for any smooth compactification $f:X\to B$ of any family 
  $(f_U:U\to V) \in \mcS$.
  Here for any divisor $\Delta$ and line bundle $\sL$ on $B$, we
  define $\deg_H(\Delta)\leteq \Delta\cdot H^{d-1}$ and
  $\deg_H(\sL)\leteq c_1(\sL)\cdot H^{d-1}$.
\end{definition}


\begin{theorem}[for a more precise version see \autoref{MAIN}]\label{thm:main}
  In the setting of \autoref{def:ARAK}, if $K_B+D$ is pseudo-effective, then all
  members of $\smn(V)$ satisfy an Arakelov inequality.
\end{theorem}

\noindent
Note that when $d=1$, our Arakelov-type inequality in \autoref{MAIN} fully recovers
the original one for curves in \eqref{ARAK}.
In fact, in \autoref{MAIN} we prove a more sophisticated and precise version.  In
particular, we prove that there exist $b'_{m,n,\nu}$, a polynomial whose coefficients are
given by explicit functions of $m$ and $\rank(f_* \omega^m_{X/B})$ (see
\autoref{MAIN} and \autoref{eq:explicit}), and an integer $\gamma_m$ such that the
polynomial $b_{m,n,\nu}$ in \autoref{ARAK2} can be written as
\[
  b_{m,n,\nu} = (\dim V)^{\gamma_m} \bdot b'_{m,n,\nu}.
\]
Note that the dependence of $b_{m,n,\nu}$ on $\gamma_m$ disappears in the case $d=1$, so
in that case $b_{m,n,\nu}$ is explicitly computable.

In addition, the integer $\gamma_m$ is an upperbound for an invariant for members of
$\smn(V)$, which we will refer to as \emph{Viehweg number}
(cf.~\autoref{def:V}).

Further note that for a non-isotrivial smooth family $f_U: U \to V$ of curves of
genus at least $2$ over a quasi-projective curve $V$, by Arakelov and Parshin's
resolution of the Shafarevich hyperbolicity conjecture, we know that $K_B+D$ is
effective (same is true when fibers are canonically polarized manifolds by
\cite{Kovacs00a}, \cite{Kovacs02}).  Therefore, the pesudo-effectivity of $K_B+D$ in
\autoref{thm:main} is a natural assumption. In fact, for families in $\smn(V)$ with
maximal variation (or equivalently those with a generically finite moduli morphism
$\mu:U\to M_{n.\nu}$ to the coarse moduli scheme $M_{n,\nu}$, see
\autoref{sect:Section4-invariant} for more details), by the culmination of the works
of \cite{VZ02}, \cite{KK08}, \cite{KK10}, \cite{MR2871152}, \cite{CP16} on Viehweg's
hyperbolicity conjecture, we know that $\kappa(B,D) = \dim B$. In particular, we have
the following direct consequence of \autoref{thm:main}:

\begin{corollary}
  Let $(B,D)$ be a smooth compactification of a smooth quasi-projective variety $V$
  (as in \autoref{def:ARAK}).  Then each member $f_U$ of the subclass of
  $\smn(V)$ consisting of the families of maximal variation satisfies an
  Arakelov inequality as in \autoref{ARAK2}.
\end{corollary}

Finally note that when $f$ is semistable, the inequality in \autoref{thm:main} can be
sharpened by replacing $\deg_H(D)$ by zero, but because this case is not the focus of
the current paper, we omit additional references and details.

\subsection{Bounding heights for substacks of stable varieties}
By using fundamental properties of the moduli stack of stable curves
Arakelov and Bedulev-Viehweg introduced a notion of a height function
on canonically polarized varieties $X$ of dimension $n$, with fixed canonical volume
$\nu:= K_X^n\in \bQ$.  That is, given any morphism $\mu: V \to M_{n,\nu}$ arising
from a smooth family over the curve $V$, let $\Phi: B'\to M_{n,\nu}$ be its
KSB-stable 
closure via a finite surjective morphism $\gamma: B' \to B$.  For sufficiently large
$m$, there is an ample line bundle $\lambda_m$ on $M_{n,\nu}$ and a positive integer
$p_m$ such that $\deg(\Phi^* \lambda_m)$---which we think of as a height function
associated to $\mu$---has an upperbound by
$p_m \cdot \deg \gamma\cdot b_{m,n,\nu}\big( g(B) , \deg(D) \big)$.

Similarly, in higher dimensions we can think of $\vol(\Phi^* \lambda_m)$ as a height
function which can be uniformly bounded using \autoref{thm:main} by the numerical
properties of $(B,D)$. {In fact, following Arakelov, one may go further and divide
  $\deg\gamma$ by $\deg\Phi$ to get a bound on $\vol(\lambda_m)$ on the image of
  $\mu$, which, as $\deg \Phi \geq \deg \gamma$, is independent of $\deg \gamma$.}

\begin{theorem}\label{thm:height}
  Using the notation introduced in \autoref{def:ARAK}, we have that for a
  sufficiently large $m\in \bN$ there exists a function
  $c_m=c_{m,n,\nu} \in \bZ[x_1, x_2, x_3]$, depending only on $m$, $n$, and $\nu$,
  \hmarginpar{\tiny \color{red} \bf S: This probably depends on the dimension as
    well,
    no? \\
    \color{blue} B: Now that we are using volume, yes. Looks like you have added
    it. }%
  such that, for every $\mu: V\to M_{n,\nu}$, arising from some
  $(f_U:U\to V)\in \smn(V)$, and the associated compactification
  $\Phi: B' \to M_{n,\nu}$, we have \hmarginpar{\tiny \color{red} \bf S:
    Why is this a dasharrow? Isn't this a morphism as above? \\
    \color{blue} B: Before I was only thinking about using stable reductions for
    these which only works in codim one. But as you have pointed out Kollár in his
    books seems to be saying that if one uses Abremov-Karu's reduction --instead of
    ss one-- then the stable family can be globally extended. I have not checked the
    details of this exactly works but it has no effect on our paper, so I have
    replaced the rational map by a morphism.}%
  \[
    \vol(\Phi^*\lambda_m) \leq (\deg\gamma )^d\bdot c_m \big( (K_B+D)\cdot H^{d-1} ,
    D\cdot H^{d-1} , H^d \big) \in \bQ_{>0} .
  \]
\end{theorem}

\subsection{Outline of the proof}\label{subsec:outline-proof}
As we will see in \autoref{sect:Section5-APineq}, using results from
\cite{CP16}, \cite{Taji18} and others, one can establish a naive, $f_U$-dependent,
upperbound for 
a smooth projective family $f_U$ in the form of \autoref{ARAK2} (see
\autoref{eq:Arineq}), as soon as $K_B+D$ is pseudo-effective.
However, in order to obtain an inequality where this upperbound is independent of the
choice of the family---the aim of an Arakelov-type inequality---one needs a more
careful approach.  We start by defining a local system via a prescribed global
section $s$ of line bundle $\sM$,
which can be naturally defined on any compactification of a smooth (or stable)
projective family $f$. It turns out that to obtain the desired Arakelov-type
inequality, it is sufficient to find a uniform bound for the rank of this local
system.  We denote this rank by $\alpha_{s}(f)$ and refer to it as the \emph{Viehweg
  number} of $f$ (cf.~\S\S\ref{subsect:Vnumbers}).

The problem of proving an Arakelov inequality over higher dimensional base spaces is
then reduced to establishing the existence of a suitable global section $s$ for which
$\alpha_{s}(f)$ has an upperbound that does not depend on $f$, but only on fixed
invariants. To achieve this, we introduce what we call the \hmarginpar{\tiny S: I
  don't think ``augmented'' is the right word for this. I changed it to ``twsited''
  for now, but we can easily change it back or to something else. Look at the def of
  in the main file. \\
  \textbf{B: Sounds good. I have now changed all to twisted}.}%
\emph{\aug direct image sheaf}. 
This is defined on $B$ and simultaneously encodes information about
$f_*\omega^m_{X/B}$ (for an appropriate $m\in\bN$), its determinant, and the
semistable locus of $f$. This sheaf is closely related to the above $\sM$ and, as we
will show in \autoref{sect:Section3-Prelim}, it is weakly positive (see
\autoref{prop:wp}).  The latter property is of particular importance for the
construction of $s$ in \autoref{thm:uniform}.

\subsubsection{Deformation spaces of families of canonically polarized manifolds}
A key component of our argument is based on the results of
\cite{KoL10} on finiteness of deformation classes for members of $\smn(V)$. That is,
by \cite{KoL10}, there is a finite subset
$\{ f_i \}_{1\leq i \leq k} \subset \smn(V)$ such that for any $f_U\in \smn(V)$,
there is some $1\leq i \leq k$, a connected scheme $W$ and a projective family
$f_W: U_W \to W\times V$ of canonically polarized manifolds such that
$(U_W)_{\{ w\} \times V} \cong_{V} U_i$ and $(U_W)_{\{ w'\} \times V} \cong_{V} U$,
for some closed points $w, w' \in W$, i.e., up to an isomorphism $U_W$ pulls back to
$U_i$ and $U$.

In fact, \cite{KoL10} goes further by showing that there is a finite type substack of
the stack of canonically polarized manifolds over a finite type scheme of the form
$W\times V$ whose connected components \emph{parametrize} members of $\smn(V)$ (see
Subsection~\ref{subsect:Par} for more details).  We use this latter result in two
ways; to find a suitable section $s_m$ for all members of $\smn(V)$, with $m$ being 
sufficiently large and divisible, and to use the
(generic) deformation-invariance property of $\alpha_{s_m}(f)$ to establish an
upperbound for $\alpha_{s_m}(f)$ that is independent of the choice of $f$, proving
the existence of $\gamma_m$ as stated in \autoref{MAIN} (see also
\autoref{thm:uniform} and \autoref{cor:KEY}).

\subsubsection{The role of stable reductions.}
As mentioned earlier, Viehweg numbers are closely related to the \aug direct image
sheaf (cf.~\S\S\ref{subsec:outline-proof})
The connection is through weak positivity of this sheaf. That is, as we will
show, up to a twist by an ample line bundle, a large enough symmetric power of this
sheaf
is generically globally generated. It turns out that the problem of bounding Viehweg
numbers can be traced to bounding the necessary exponent for this symmetric power.
Although the finiteness of deformation classes for $\smn(V)$ proved in \cite{KoL10}
is key to uniformly bounding it, as we will see in \S\S\ref{subsect:GGG}, it is not
enough for finding such bounds for compactifications of members of $\smn(V)$.  To
accomplish the latter, strict base change properties (in the sense of
\autoref{prop:BC}) are needed for the \aug direct image sheaf 
that generally only holds for KSB-stable closures via stable reductions.
\hmarginpar{\tiny \color{red} \bf S: I am not sure what you have in mind with {\it
    Such compactifications have been found by Koll\'ar \cite{Kollar10} and Hacon-Xu
    \cite{HX13}}. I'm OK with including these, I just don't see why we need them
  here. I added a reference that I think is relevant, but I may be misunderstanding
  what you are trying to say here.}%
The required compactifications exist by \cite{ModBook}*{Thm.~4.59}
(cf.~\cites{KSB88,Karu00}) and play a significant role in our strategy to establish
Arakelov inequalities over higher dimensional base spaces.

\subsection{Related results} 
\label{subsect:related}
As mentioned above, Arakelov type inequalities generally fall into two related
categories. The geometric one \eqref{ARAK} goes back to Arakelov. This was later
generalized in \cite{Bedulev-Viehweg00}. Further refinements and generalizations over
curves was established in Viehweg-Zuo \cite{VZ02}.  There are also more Hodge
theoretic Arakelov-type inequalities, which are concerned with establishing universal
bounds for the degree of direct summands of Hodge bundles underlying (canonical
extensions of) variation of Hodge structures (or VHS for short) of geometric origin.
When the fibers are of dimension $1$, then one can interpret this type of Arakelov
inequalities for VHSs of weight one as essentially the same as \eqref{ARAK}, with
$m=1$.  Such inequalities were initiated by Deligne \cite{Deligne87} and later
extended by Peters \cite{Pet00} and Jost-Zuo \cite{JZ02}. All of these results are
restricted to the case when the base of the family is $1$-dimensional.  Using
Simpson's nonabelian Hodge theory, under rather strong positivity assumptions for
$\Omega^1_X(\log B)$, some Arakelov-type inequalities for VHSs of weight one were
generalized to higher dimensional base spaces in \cite{VZ07}. Topological
counterparts of these inequalities were also studied by Bradlow, Garc{\'i}a-Prada,
and Gothen \cite{BGG06} and by Koziarz and Maubon \cite{KM08a}, \cite{KM08b}. A
detailed review of Arakelov inequalities can be found in \cite{Vie08}.  We also refer
to the paper of Brotbek and Brunebarbe \cite{BB20} for some other recent developments
in this area.

\subsection*{Acknowledgements}
We thank Sho Ejiri and Sung Gi Park for helpful comments.

\section{Preliminaries and Background}
\label{sect:Section2-background}

\begin{definition}\label{def:families-pairs}
  In this article a \emph{variety} means a reduced, finite type scheme, and every
  scheme will be assumed to be defined over $\bC$.  A \emph{pair} $(X,\Delta)$
  consists of a variety $X$ and an effective $\bQ$-Weil divisor
  $\Delta = \sum a_i \Delta_i$, $a_i\leq 1$. We say $(X,\Delta)$ is a \emph{reduced
    pair} if $a_i=1$, for each $i$. 
  An \emph{snc pair} is a reduced pair $(X,\Delta)$ such that $X$ is regular and
  $\Delta$ is an snc divisor. 
  A \emph{morphism of reduced pairs} $f:(X,\Delta)\to (B,D)$ is a dominant morphism
  $f:X\to B$ of schemes with connected fibers such that
  $ f^{-1}(\supp D)\subseteq \supp\Delta$.
  Assuming that $D$ is $\bQ$-Cartier, we will use the notation
  $f^{-1}D\leteq (f^*D)_{\mathrm{red}}$ to denote the \emph{reduced preimage} of $D$.
  Using this notation, the above criterion can be replaced by $\Delta \geq f^{-1}D$.
  A \emph{morphism of snc pairs} is a morphism of reduced pairs
  $f:(X,\Delta)\to (B,D)$ such that both $(X,\Delta)$ and $(B,D)$ are snc pairs.
\end{definition}

\begin{definition}\label{def:strongly-ample}\cite{ModBook}*{8.34}
  Let $X$ be a proper scheme and $\sL$ a line bundle on $X$.  $\sL$ is said to be
  \emph{strongly ample} if it is very ample and $H^i(X, \sL^q)=0$ for $i,q>0$.
  Note that by \cite[I.8.3]{Laz04-I}, if this holds for all $q\leq \dim X+1$ then it
  holds for all $q>0$. In particular, strong ampleness is an open condition in flat
  families.
  
  Similarly, let $f\colon X\to B$ be a proper, flat morphism and $\sL$ a line bundle
  on $X$.  We say that $\sL$ is \emph{strongly $f$-ample} or \emph{strongly ample
    over $B$}, if $\sL$ is strongly ample on the fibers. Equivalently, if $\sL$ is
  $f$-very ample and $\myR^if_*\sL^q=0$ for $i,q>0$. It follows that in this case
  $f_*\sL$ is locally free and we get an embedding $X\into \bP_B(f_*\sL)$.
  We will be mainly interested in the case when $f\colon X\to B$ is stable and
  $\sL=\omega_{X/S}^{[q]}$ for some $q>0$. In this case, if $q>1$ then
  $\myR^if_*\sL^m=0$ for $i,m>0$ by \cite{ModBook}*{11.34}.
\end{definition}

\begin{definition}[Snc and \reduced morphisms]\label{def:snc-strong}
  Consider a morphism $f:X\to B$ and a decomposition $\Delta=\Delta_v+\Delta_h$ into
  vertical and horizontal parts, i.e., such that $\codim_Bf(\Delta_v)\geq 1$ and that
  $f|_{\Delta_0}$ dominates $B$, for any irreducible component $\Delta_0\subseteq
  \Delta_h$. Using this decomposition, we call a morphism of snc pairs
  $f:(X,\Delta) \to (B, D)$ an \emph{snc morphism}, if $f$ is flat,
  $\Delta_v=f^{-1}D$ and $f|_{X\setminus \Delta_v}$ is smooth.
  
  An snc morphism $f:(X,\Delta)\to (B,D)$ is called \emph{\reduced}, if $f^*D$ is
  reduced. Note that this implies that then $\Delta_v=f^*D$.  Further note that an
  snc morphism with reduced fibers is necessarily \reduced. Semistable (see
  \cite{Abramovich-Karu00}*{0.1} for a definition) and stable snc morphisms
  (\autoref{def:stabfam}) are the main examples to which this will be applied.
\end{definition}


\begin{notation}\label{not:det}
  Given a reduced scheme $X$ and a coherent sheaf $\sF$ of rank $r$, for any
  $m\in \bN$, we define $\sF^{[m]} : = \big( \sF^{\otimes m} \big)^{**}$, where
  $(\sblank)^{**}$ denotes the double dual. We will apply the same notation for all
  tensor operations. In particular, $\Sym^{[m]}\sF\leteq \left(\Sym^m\sF\right)^{**}$
  and $\bigwedge^{[m]}\sF\leteq \left(\bigwedge^m\sF\right)^{**}$.  Furthermore,
  $\det\sF$ will denote $\bigwedge^{[r]}\sF=\big(\bigwedge^r \sF\big)^{**}$. Notice
  that if $X$ is regular, then $\det\sF$ is a line bundle.
\end{notation}

\begin{notation}\label{not:rel-Omega}
  Let $f: (X, \Delta) \to (B,D)$ be an snc morphism. Consider the natural morphism
  $\eta: f^*\Omega^1_B(\log D) \to \Omega^1_X(\log \Delta)$ and define the \emph{sheaf of
    relative log differentials} as the cokernel of this morphism:
  $\Omega^1_{X/B}(\log\Delta)\leteq \coker\eta$.  In other words, there exists a \ses:
  \[
    \xymatrix{%
      0 \ar[r] & f^*\Omega^1_B(\log D) \ar[r] & \Omega^1_X(\log \Delta) \ar[r] &
      \Omega^1_{X/B} (\log \Delta) \ar[r] &  0.
    }
  \]
  The first two sheaves are locally free by definition and a simple local calculation
  shows that so is the third one. In particular, the exterior powers of this sheaf,
  denoted by
  $\Omega_{X/B}^p (\log \Delta)\leteq \bigwedge^p \Omega^1_{X/B} (\log \Delta)$ for
  $p\in\bN$, are also locally free and if $\dim X=d$ and $\dim B=r$, then we have the
  following isomorphism:
  \begin{equation}
    \label{eq:3}
    \Omega^{d-r}_{X/B} (\log \Delta) \simeq \Omega_X^d(\log \Delta)\otimes
    \left(f^*\Omega_B^r(\log D)\right)^{-1}\simeq 
    \omega_{X/B} (\Delta - f^* D)
  \end{equation}
  Observe that if $f$ is \reduced, then (using the notation from
  \autoref{def:snc-strong}) the last sheaf is isomorphic to $\omega_{X/B}(\Delta_h)$.
\end{notation}

\begin{notation}\label{def:canonical}
  \hmarginpar{\tiny B: I added dualizing here since I thought we need to refer to it
    in Section 4 but strictly speaking we don't. Although I still think it helps
    clarifying the notation, especially with different references.  If you are ok
    with this, PLEASE CHECK, \textbf{especially the isomorphism
      $\omega_{X/B} \simeq \sO_X(K_X- f^* K_B)$}. \newline S: I don't like to call
    this the ``dualizing'' sheaf, because if $X$ is not CM, then it is not
    dualizing. \\ Also: I added a reference for that isom.}%
  For a morphism of finite type $f: X\to B$ of relative dimension $n$ we define the
  relative canonical sheaf by $\omega_{X/B} : = h^{-n}\big( f^! \sO_B \big)$.  For an
  $S_2$ and $G_1$ (Gorenstein in codimension one) scheme $X$, a canonical divisor is
  denoted by $K_X$ (see \cite[\S5]{Kov13} for more details).  If $X$ is $S_2$ and
  $G_1$ and $B$ is Gorenstein,
  then $\omega_{X/B} \simeq \sO_X(K_X - f^*K_B)$ by \cite[(3.3.6)]{Conrad00}.
\end{notation}

\hmarginpar{\tiny S: I commented out the notation defining ``the'' Hilbert polynomial
  as we probably don't need it anymore.}%

In this paper we only need the following slightly restrictive definition. 

\begin{definition}[Stable families]\label{def:stabfam}
  A projective variety $Z$ is called \emph{stable}, if it has slc singularities
  \cite[1.41]{ModBook} and $\omega_Z$ is an ample $\bQ$-line bundle.
  Let $B$ be a reduced scheme.  A projective morphism $f:X\to B$ is called
  \emph{stable}, if $X_b$ is a stable variety for each $b\in B$ and
  $\omega^{[m]}_{X/B}$ is invertible, for some $m\in \bN$ cf.~\cite{ModBook}*{3.40}.
\end{definition}

\begin{remark}
  Note that if $B$ is not assumed to be reduced or if one considers families of
  pairs, then the definition of a stable family is more complicated. The fact that in
  this case the above definition suffices follows from
  \cite{ModBook}*{4.7}. 
\end{remark}

\begin{notation}[Pullback]\label{not:pullback}
  Given morphisms of schemes $f: X\to B$ and $Z\to B$ we denote the pullback of $f$
  by $f_Z:X_Z=X\times_B Z\to Z$ and $\pr:X_Z\to X$ denotes the induced natural
  projection.
\end{notation}

\begin{remark}[Base change properties]\label{rem:BaseChange}
  Let $f:X\to B$ be a family with slc fibers. Then the relative canonical sheaf of
  $f$ is flat over $B$ with $S_2$ fibers and compatible with arbitrary base change by
  \cites{MR2629988,MR4156425}, cf.~\cite{ModBook}*{2.67}.
  
  \hmarginpar{\tiny S: I added ``with a normal general fiber''}%
  For a stable family $f:X\to B$
  , the formation of $\omega^{[m]}_{X/B}$ commutes with arbitrary base change for
  every $m\in \bN$, by
  \cite{ModBook}*{4.33}, cf.~\cite{MR4059993}*{Prop.~16},
  \hmarginpar{\tiny B: Is $Z$ reduced enough or do we need more assumptions on $Z$? \\
    S: reduced is enough, but I've been having a problem finding a good reference. I
    might have to fix this later (after submission). I added references that should
    be OK for submitting and we will fix this before publication.}%
  that is, for any reduced scheme $Z$ and morphism $\psi: Z\to B$, and any
  $k\in \bN$, we have that $\psi_X^* \omega^{[k]}_{X/B}\simeq \omega^{[k]}_{X_Z/Z}$,
  and that the isomorphism
  \begin{equation}
    \label{eq:1}
    \psi^* f_* \omega^{[m]}_{X/B}  \simeq  (f_Z)_* \omega^{[m]}_{X_Z/Z}
  \end{equation}
  holds when $\omega^{[m]}_{X/B}$ is strongly $f$-ample
  (cf.~\autoref{def:strongly-ample}), e.g., for all sufficiently large and divisible
  $m$.

  Note that a stable family as defined in \autoref{def:stabfam} is KSB-stable in the
  sense of \cite[Def.~6.16]{ModBook}, cf.~\cite{ModBook}*{4.33}.
\end{remark}

\begin{notation}[Discriminant locus]\label{not:disc}
  Let $f:X\to B$ be a dominant morphism of regular schemes. Denote the divisorial
  part of the discriminant locus of $f$ by $\disc (f)$. Setting $D_f= \disc(f)$, we
  let $\Delta_f\leteq f^{-1}D_f$, a reduced divisor on $X$.  This way the resulting
  map \hmarginpar{\tiny B: I changed to equality to avoid mentioning this again. I
    think when we use this notation, we are always in the equality setting.}%
  $f: (X, \Delta_f) \to (B,D_f)$ is a morphism of reduced pairs.  If in addition,
  $f: (X, \Delta_f) \to (B,D_f)$ is \reduced, then $\Delta_f = f^*D_f$ and if
  $\dim X/B =n$, then there is an isomorphism (cf.~\autoref{eq:3})
  \begin{equation}\label{ssCase}
    \Omega^n_{X/B} (\log \Delta_f) \simeq \omega_{X/B} (\Delta_f - f^* D_f) \simeq
    \omega_{X/B} . 
  \end{equation}
\end{notation}

\noin We define a similar notion for morphisms of arbitrary schemes.

\begin{notation}[Non-reduced locus]\label{not:non-reduced}
  Let $f:X\to B$ be a morphism of schemes and denote the divisorial part of the locus
  of non-reduced fibers on $B$ by $R_f$, i.e., let
  \[
    R^+_f\leteq \left\{ b\in B \skvert X_b \text{ is not reduced} \right\},
  \]
  and let $R_f$ be the reduced divisor corresponding to the union of those
  irreducible components of $R^+_f$ that are codimension one in $B$.
  
  Note that if $f$ is an snc morphism of relative dimension $n$, then by
  \hmarginpar{\tiny B: Isomorphism \ref{eq:7} should be an inclusion.\\ \color{blue}
    S: Corrected. \\
    \color{red} B: Thanks!}%
  \autoref{eq:3},
  \begin{equation}
    \label{eq:7}
    \Omega^{n}_{X/B} (\log \Delta_f) \simeq 
    \omega_{X/B} (f^{-1}R_f-f^*R_f)
  \end{equation}
  Further note that if $f$ is strongly snc or stable, then $R^+_f=\emptyset$ and
  hence $R_f=0$.
\end{notation}

\begin{lemma}\label{lem:tau-injects}
  Let $f:X\to B$ be a dominant morphism of regular schemes and let $\tau:\wt X\to X$
  be a projective birational morphism and let $\wt f\leteq f\circ \tau$.  Assume that
  $\Delta_f$ and $\Delta_{\wt f}$, defined in \autoref{not:disc}, are snc divisors
  and that $\tau$ is an isomorphism outside $\Delta_f$.  Then, over the locus where 
  $f$ and $\widetilde f$ are strongly snc, there exists an
  injective morphism
  $\tau_*\Omega^n_{\wt X/B} (\log \Delta_{\wt f}) \into \Omega^n_{X/B} (\log
  \Delta_f)$.
\end{lemma}

\begin{proof}
  First, observe that $D_f\subseteq D_{\wt f}$ and hence
  ${\tau}^{-1}\Delta_f\subseteq \Delta_{\wt f}$.  Further we note that 
  by construction we have 
  \[
    \Delta_{\wt f} \leq \tau^*\Delta_f .
  \]
  On the other hand, as $X$ is non-singular, there exists another $\tau$-exceptional
  effective divisor $E_1$ such that $\omega_{\wt X}\simeq \tau^*\omega_X(E_1)$.
  Putting everything together we obtain that
  \[
    \omega_{\wt X/B}(\Delta_{\wt f}-\wt f^*D_{\wt f}) \subseteq
    \tau^*\left(\omega_{X/B}(\Delta_f-f^*D_f)\right)\otimes \sO_{\wt X}(E_1),
  \]
  and hence
  \[
    \tau_*\omega_{\wt X/B}(\Delta_{\wt f}-\wt f^*D_{\wt f}) \subseteq
    \omega_{X/B}(\Delta_f-f^*D_f)\otimes \tau_*\sO_{\wt X}(E_1).
  \]
  By \cite{KMM87}*{1-3-2} $\tau_*\sO_{\wt X}(E_1)\simeq \sO_X$ and hence the
  above containment combined with \autoref{eq:3} implies the desired statement.
\end{proof}

\begin{notation}\label{not:fiberprod}
  For a morphism of normal schemes $f:X\to B$, for any $r\in \bN$, we denote the
  $r$-fold fiber product by
  \[
  X^r := \underbrace{X \times_B \ldots \times_B X}_{\text{$r$ times}} ,
  \]
  with the induced morphism $f^r: X^r \to B$.  Furthermore, let $\pi:X^{(r)}\to X^r$
  denote a strong resolution
  of $X^r$, with the naturally induced map $f^{(r)}: X^{(r)}\to B$. 
  ({Recall that a strong resolution $\pi: Y\to X$ is a resolution 
    of $X$ for which $\pi|_{\pi^{-1}(X_{\reg})}$ is an isomorphism, where $X_{\reg}$
    denotes the regular locus of $X$.})%
  \hmarginpar{\tiny S: I moved the footnote here to a parenthetical sentence}
  \hmarginpar{\tiny S: I think we should use either ``resolution'' or
    ``desingularization'', and definitely not using one to define the other... \\
    \color{red} B: Sounds good. I have left the global change to you. :-) }%
  \hmarginpar{\tiny S: What's a strong resolution here? \\ \textbf{B: There were
      multiple typos here. Is it OK now?} \\ S: ``strong resolution'' is still not
    defined... }%
  \hmarginpar{\tiny S: I would like to revisit this notation, but we can do this
    after it is posted.}%
\end{notation}

\begin{proposition}\label{prop:rfold}
  Given a stable family $f:X\to B$, $f^r$ is also stable. Furthermore, for every $m$
  for which $\omega^{[m]}_{X/B}$ is invertible, $\omega^{[m]}_{X^r/B}$ is also
  invertible and, over the complement of a subscheme of $B$ of $\codim_B\geq 2$, we
  have
 \begin{equation}\label{rfold}
   f^r_* \omega^{[m]}_{X^r/B}  \simeq \bigotimes^r f_*\omega^{[m]}_{X/B} .
 \end{equation}
\end{proposition}
 
\begin{proof}
  An iterated application of \cite{MR3143705}*{Prop.~2.12} shows that $f^r$ is
  stable. 
  For the rest of the claim we use induction on $r$.  Observe that
  $X^r=X\times_BX^{r-1}$ and let $\pr:X^r\to X^{r-1}$ 
  and $\pr_r: X^r\to X$ denote the natural projections.
  \[
    \xymatrix{%
     \ \ X^r \ar[r]^{\pr} \ar[d]_{\pr_r}  & X^{r-1} \ar[d]^{f^{r-1}} \hskip-1.65em\\
      X \ar[r]_f & B
    }
  \]
  Then
  $\omega^{[m]}_{X^r/X} \simeq \pr^*\omega^{[m]}_{X^{r-1}/B}$ by
  \autoref{rem:BaseChange} and hence
  $(\pr_r)_* \omega^{[m]}_{X^r/X}\simeq f^* f^{r-1}_* \omega_{X^{r-1}/B}^{[m]}$ by
  flat base change.  Applying $f^r_*=f_*(\pr_r)_*$ to
  \begin{equation*}
    \omega^{[m]}_{X^r/B}  \simeq  \omega^{[m]}_{X^r/X} \otimes \pr_r^*
    \omega^{[m]}_{X/B},
  \end{equation*}  
  and using the induction hypothesis then yields
  \begin{equation*}
    f_*(\pr_r)_* \omega^{[m]}_{X^r/B}
    \simeq f_*\left( f^* f^{r-1}_* \omega_{X^{r-1}/B}^{[m]}
      \otimes\omega^{[m]}_{X/B}\right) \simeq \left(\bigotimes^{r-1}
      f_*\omega^{[m]}_{X/B}\right)\otimes f_*\omega^{[m]}_{X/B}. \qedhere
  \end{equation*}
\end{proof}

\begin{corollary}\label{prop:inj-for-stable}
  Let $f:X\to B$ be a family which is stable in codimension one.  Assume that $B$ is
  quasi-projective and that $\omega^{[m]}_{X/B}$ is invertible
  in codimension one.  After removing a subscheme of $B$ of $\codim_B\geq 2$ if
  necessary, there exists an injection
  \begin{equation}\label{KeyInject-stable}
    \left(\det f_*\omega_{X/B}^{[m]}\right)^{\otimes k}
    \hooklongrightarrow   
    f_*^{  k\bdot r_m } \omega^{[m]}_{X^{ k\bdot r_m  }/ B}, 
  \end{equation}
  for any $k \in \bN$ where $r_m: = \rank(f_* \omega^{[m]}_{X/B})$.
\end{corollary}

\begin{proof}
  By removing a subscheme of $B$ of $\codim_B\geq 2$ if necessary we may assume that
  $f$ is stable, $\omega^{[m]}_{X/B}$ is invertible and
  that $f_* \omega^{[m]}_{X/B}$ is locally free on $B$.
  Raising the natural embedding 
  \[
    \det f_* \omega^{[m]}_{X/B} \hooklongrightarrow \bigotimes^{r_m} f_*
    \omega^{[m]}_{X/B}
  \]
  to the $k^{\text{th}}$ power yields
  \[
    \big( \det f_* \omega^{[m]}_{X/B}\big)^{\otimes k} \hooklongrightarrow
    \bigotimes^{ k \bdot r_m} f_* \omega^{[m]}_{X/B} \simeq f_*^{ k\bdot r_m }
    \omega^{[m]}_{X^{ k\bdot r_m }/ B},
  \]
  where the last isomorphism is simply \autoref{rfold}.
\end{proof}

\subsection{Determinants of direct image sheaves and base change}
\label{ss:imagesheaves}

\begin{def-not}\label{not:L}
  Let $f:X \to B$ be a morphism of finite type of normal schemes.  Assume that $B$ is
  regular, and fix an $m\in\bN$.  We define the sheaf $\Lm(f)$ as follows:
  \[
    \Lm(f): = \det \big( f_*\omega^{[m]}_{X/B} (-m\bdot R_f) \big) ,
  \]
  where $R_f$ is as in \autoref{not:non-reduced}.
\end{def-not}

The following is a trivial observation, but we record it so we can easily cite it
when needed.

\begin{lemma}\label{lem:Lm-contains}
  Let $f:X \to B$ be a morphism of finite type of normal schemes.  Assume that $B$ is
  regular, and fix an $m\in\bN$.
  Then $\Lm(f)$ is a line bundle and
  \begin{equation}
    \label{eq:4}
    \Lm(f) \supseteq \det \big( f_*\omega^{[m]}_{X/B} (-m\bdot D_f) \big).
  \end{equation}
  Furthermore, if $f$ has reduced fibers (e.g., it is \reduced or stable), then
  \begin{equation}
    \label{eq:5}
    \Lm(f)=    \det  f_*\omega^{[m]}_{X/B}.
  \end{equation}
\end{lemma}

\begin{lemma}\label{lem:W-for-snc}
  Let $f:(X,\Delta) \to (B,D)$ be an snc morphism of relative dimension $n$.  Fix an
  $m\in\bN$. Then (using \autoref{not:disc}), 
  \[
    \Lm(f)\subseteq \det f_* \left(\Omega^{n}_{X/B} (\log\Delta_f)^{\otimes m}
    \right).
  \]
\end{lemma}

\begin{proof}
  This follows directly from \autoref{eq:7}.
\end{proof}

\hmarginpar{\tiny B: Have used snc morphism to clean up this lemma.}%
\begin{lemma}\label{lem:injection}
  Let $(X, \Delta)$ and $(B,D)$ be two reduced pairs.  Assume that $(B,D)$ is snc.
  Let $f: (X, \Delta) \to (B, D)$ be a morphism of reduced pairs, with
  $\dim X/B =n \neq 0$. Further let $B'$ be a regular variety and $\gamma: B' \to B$
  a flat surjective morphism.  Let $\pi:Y\to X\times_BB'$ be a resolution of
  singularities, $D'=(\gamma^*D)_{\red}$, and $\Sigma$ a reduced divisor on $Y$ such
  that $g=f'\circ\pi: (Y, \Sigma) \to (B', D')$ is an snc morphism.  These objects
  and morphisms fit in the following commutative diagram of morphisms of pairs:
  \begin{equation*}
    \xymatrix{
      (Y, \Sigma) \ar@/^7mm/[rrrr]^{\mu} \ar[rr]^-{\pi}  \ar[drr]_{g} &&
      (X _ {B'}, \Delta_{B'})
      \ar[rr]_-{\gamma'} \ar[d]^{f'} &&  (X, \Delta) \ar[d]^{f} \\   
      &&        (B', D')  \ar[rr]^{\gamma}    && (B,D) .
    }
  \end{equation*}

  \begin{enumerate}
  \item\label{ONE} If $f: (X, \Delta) \to (B, D)$ is an snc morphism, then there is a
    natural injective morphism 
    \[
      \gamma^* f_* \big( \Omega^n_{X/B} (\log \Delta)^{\otimes m} \big)
      \hooklongrightarrow g_* \big( \Omega^n_{Y/B'} (\log \Sigma) ^{\otimes m} \big)
    \]
    which is generically an isomorphism over $B'$. 
  \item\label{TWO} If $g$ is \reduced, then for every projective birational morphism
    $\eta: \wtilde X\to X$ which is an isomorphism outside $\Delta_f$ and such that
    (using \autoref{not:disc})
    $\wtilde f:= f \circ \eta : (\wtilde X, \Delta_{\wtilde f}) \to (B, D_{\wtilde
      f})$ is an snc morphism, there exists a natural injection
    \[
    \gamma^* \Lm(\wtilde f) \hooklongrightarrow \det g_* \omega^m_{Y/B'} ,
    \]
    which is generically an isomorphism over $B'$.
  \end{enumerate}
\end{lemma}

\begin{proof}
  First, note that there is a natural injective morphism 
  \[
    \mu^* \Omega^n_{X/B} (\log \Delta)^{\otimes m} \hooklongrightarrow
    \Omega^n_{Y/B'} ( \log \Sigma)^{\otimes m},
  \]
  and hence another one
  \begin{equation}\label{1}
    f'_*\pi_*\pi^*  (\gamma')^* \Omega^n_{X/B} (\log \Delta)^{\otimes m}=
    g_*  \mu^* \Omega^n_{X/B} (\log \Delta)^{\otimes m}
    \hooklongrightarrow  
    g_* \Omega^n_{Y/B'} (  \log \Sigma)^{\otimes m},
  \end{equation}
  which is an isomorphism over a dense open subset of $B'$.

  On the other hand, 
  there exists a natural morphism, 
  \[
    (\gamma')^* \Omega^n_{X/B} (\log \Delta)^{\otimes m} \longrightarrow 
    \pi_* \pi^* (\gamma')^* \Omega^n_{X/B} (\log \Delta)^{\otimes m} %
  \]
  which is
  also an isomorphism over the preimage of a dense open subset of $B'$.
  This morphism, combined with the one in \autoref{1} gives a morphism
  \begin{equation*}
    f'_*(\gamma')^* \Omega^n_{X/B} (\log \Delta)^{\otimes m} \longrightarrow  
    f'_*\pi_*\pi^*  (\gamma')^* \Omega^n_{X/B} (\log \Delta)^{\otimes m}\longrightarrow  
    g_* \Omega^n_{Y/B'} (  \log \Sigma)^{\otimes m},
  \end{equation*}
  which is again an isomorphism over a dense open subset of $B'$.  By flat base
  change, the left hand side is isomorphic to
  $\gamma^* f_* \big( \Omega^n_{X/B} (\log \Delta)^{\otimes m} \big)$ and hence
  \autoref{ONE} follows.
  
  For \autoref{TWO}, 
  eliminate the points of indeterminacy of the birational map
  $Y \dashrightarrow \wt X_{B'}$, 
  let $\wtilde Y$ denote a resolution of singularities of the result and
  $\tau : \wtilde Y \to Y$ \hmarginpar{\tiny S: replaced $\pi$ with $\tau$, because
    $\pi$ was already used for another map. Check that I've got all the occurrences
    correctly! \bf{B: Done. All OK. }} %
  the induced projective birational morphism. We may assume that the induced morphism
  $\wtilde g: (\wtilde Y, \Delta_{\wtilde g})\to (B', D_{\wtilde g})$ is snc (here we
  are using \autoref{not:disc}), after removing a subset of $B'$ of
  $\codim_{B'} \geq 2$, if necessary. \hmarginpar{\tiny S': why do we need to remove
    anything? Couldn't we just do a further desingularization? \textbf{B: the snc
      condition was supposed to be on the morphism $\wtilde g$ and not
      $(\wtilde Y, \Delta_{\wtilde g})$.} \\ S; Uh! OK. }%
  We thus have the following commutative diagram:
  \[
    \xymatrix{ \wtilde Y \ar[d]_{\tau} \ar[rr] \ar[drr]^{\wtilde g} && \wtilde X_{B'}
      = \wtilde X \times_B B'
      \ar[rr] \ar[d] &&  (\wtilde X, \Delta_{\wtilde f}) \ar[d]^{\wtilde f} \\
      Y \ar[rr]^{g}_{\text{strong snc}} && B' \ar[rr]^{\gamma} && (B, D_{\wtilde f}).
    }
  \]
  According to \autoref{ONE} there is an injection
  \begin{equation}\label{INJ}
    \gamma^* \wtilde f_* \big(  \Omega^n_{\wtilde X/B} (\log \Delta_{\wtilde
      f})^{\otimes m} \big) \hooklongrightarrow 
    \wtilde g_* \Omega^n_{\wtilde Y/B'} (\log \Delta_{\wtilde g})^{\otimes m}.
  \end{equation}
  Moreover we have the following isomorphisms and containment:%
  \hmarginpar{\tiny S: I suppose the last isomorphism holds because $g$ is
    semistable. However, I think most people will be baffled by this, so we should
    add a lemma that says something like this (or more generally, without the
    semistable assumption, like in the next displayed row). If we're short on time,
    then we can do this after posting the first version to the arXiv. \bf{B: I have
      included an explanation.}}%
  \begin{equation}\label{eq:SSIsom}
    \begin{multlined}
      \hskip-2em \wtilde g_* \big( \Omega^n_{\wtilde Y/B'} (\log \Delta_{\wtilde
        g})^{\otimes m} \big) \simeq g_* \tau_* \big( \Omega^n_{\wt Y/B'} (\log
      \Delta_{\wtilde g})^{\otimes m} \big)
      \overset{~\autoref{lem:tau-injects}}{\into} \\ \into g_* \big( \Omega^n_{Y/B'}
      (\log \Delta_g)^{\otimes m} \big) \overset{\autoref{ssCase}}{\simeq}
      g_*\omega^m_{Y/B'} . \hskip -5em
    \end{multlined}
  \end{equation}
  Combining \autoref{INJ} and \autoref{eq:SSIsom} and taking determinants implies
  \autoref{TWO}.
\end{proof}

\begin{cor}
  Under the assumptions and notation of \autoref{lem:injection} and \autoref{TWO},
  there exists a natural injective morphism,
  \[
    \gamma^* \det \left( \big( \wtilde f_* \omega^m_{\wtilde X/B} \big) (- m\bdot
    D_{\wtilde f}) \right) \hooklongrightarrow \det g_* \omega^m_{Y/B'}.
  \]
  Furthermore, if in addition $\wt f$ is \reduced, then there exists a natural
  injective morphism,
  \[
    \gamma^*( \det \wtilde f_* \omega^m_{\wtilde X/B} ) \hooklongrightarrow \det g_*
    \omega^m_{Y/B'} .
  \] 
\end{cor}

\begin{proof}
  This follows directly from \autoref{lem:Lm-contains} and \autoref{TWO}.
\end{proof}


\begin{proposition}\label{prop:injection}
  In the situation of the \autoref{not:L}, assume that $X$ and $B$ are regular and quasi-projective. 
  After removing a subscheme of $B$ of $\codim_B\geq 2$ if necessary, there exists an
  injection 
  \begin{equation}\label{KeyInject}
    \Lm(f)^{\otimes k}   \hooklongrightarrow  
    f_*^{ ( k\bdot r_m ) } \omega^{m}_{X^{( k\bdot r_m )}/ B},
  \end{equation}
  for any $k \in \bN$ where $r_m: = \rank(f_* \omega^{m}_{X/B})$.
\end{proposition}

\begin{remark}\label{rem:inj-for-stable}
  Notice that by \autoref{prop:inj-for-stable}, we have that if in addition $f$ is
  stable in codimension one, then after removing a subscheme of $B$ of $\codim_B\geq 2$ if necessary,
  there exists an injection
  \begin{equation}\label{KeyInject-stable}
    \Lm(f)^{\otimes k}   \hooklongrightarrow  
    f_*^{ k\bdot r_m  } \omega^{[m]}_{X^{( k\bdot r_m )}/ B},
  \end{equation}
  for any $k \in \bN$ where $r_m: = \rank(f_* \omega^{[m]}_{X/B})$.
\end{remark}

\begin{proof}
  Without loss of generality we may assume that $(X, \Delta_f) \to (B, D_f)$ is snc.
  Let $g: Y \to B'$ be a semistable reduction of $f$ in codimension one, via the
  finite, flat, surjective and Galois morphism $\gamma: B' \to B$, (cf.~\cite{KKMS}
  and \cite{BG71}, and also \cite{Abramovich-Karu00}*{\S5}, \cite{MR4098246},
  \cite{Laz04-I}*{4.1.6}). Let $G:= \Gal(B'/B)$.  By \autoref{TWO} we have
  \hmarginpar{\tiny WARNING! This is not entirely true! \autoref{TWO} only implies
    this if $f$ is snc! \\
    B: I am not sure what the problem was but I am guessing it had something to do
    with those missing assumptions in the setting of the proposition that I have now
    added (due to the change of definition for $\sW_m$)\\
  \color{blue} S: OK, this seems to have fixed it.}%
  \begin{equation}\label{eq:LBaseChange}
    \gamma^*\Lm(f)  \hooklongrightarrow  \det g_*  \omega^m_{Y/B'}, 
  \end{equation}
  which is generically an isomorphism over $B'$.

  By raising \autoref{eq:LBaseChange} to the power $k$ we obtain the injections
  \hmarginpar{\tiny S: why the two different notation? \bf{B: I think you were
      referring to the awkward $m^a\bdot e$ (which is what appears later on in
      application). Of course there is no need for it here so I changed it to $k$.}}%
  \[
    \gamma^* \Lm(f)^{\otimes k} \hooklongrightarrow \big( \det g_*
    \omega^m_{Y/B'}\big)^{\otimes k} \hooklongrightarrow \bigotimes^{ k \bdot r_m}
    g_* \omega^m_{Y/B'} .
  \]
  On the other hand, over the semistable locus of $g$ we have
  \[
    \bigotimes^{k \bdot r_m} g_* \omega^m_{Y/B'} \simeq 
    g^{(k \bdot r_m)}_* \omega^m_{Y^{(k \bdot r_m)} /
      B'}.
  \]
  Furthermore, \cite[\S 3, p.~336]{Viehweg83} \hmarginpar{\tiny S: Can we make this
    reference more explicit?
    \\
    B. Unfortunately one needs to look inside the proof but I have added a page
    number \\ \color{blue} S: This is good enough. Thanks.}%
  implies that
  \[
    g^{(k \bdot r_m)}_* \omega^m_{Y^{(k \bdot r_m)} /
      B'} \subseteq \gamma^* f_*^{(k \bdot r_m)} \omega^m_{X^{(k
        \bdot r_m)}/B}
  \]
  so
  \begin{equation}\label{eq:pullback1}
    \gamma^* \Lm (f)^{\otimes k}   \hooklongrightarrow  \gamma^* f_*^{(k
      \bdot r_m)}  \omega^m_{X^{(k \bdot r_m)}/ B} . 
  \end{equation}
  The required injection \autoref{KeyInject} follows from applying the functor
  $\gamma_*(\sblank)^G$ to (\ref{eq:pullback1}).
  %
\end{proof}

\subsection{Positivity notions for families of varieties}
\begin{definition}\label{def:GGG}
  An $\sO_Y$-module $\sF$ on a reduced scheme $Y$ is called \emph{globally generated
    over an open subset $V\subseteq Y$}, if the natural map
  \[
    H^0(Y, \sF) \otimes \sO_{V} \longrightarrow \sF \otimes \sO_{V}
  \]
  is surjective over $V$. When the open set $V$ is not specified, we say $\sF$ is
  \emph{generically generated by global sections over $Y$}. \hmarginpar{\tiny S: I
    think ``generically generated by global sections over $Y$'' is better than
    ``generically globally generated over $Y$'', because the latter could be
    understood that restricted to a dense open set it is globally generated
    there. But that holds for everything. \bf{B: Agreed! I blame Viehweg for this!!
      We just need to remember to change this expression accordingly in Section 4.}}%
\end{definition}


Recall that given a regular quasi-projective variety $X$ and an open subset
$U\subseteq X$, a torsion free sheaf $\sF$ on $X$ is called \emph{weakly positive
  over $U$}, if $\sF|_U$ is locally free and that for any ample line bundle
$\mathcal H$ on $X$, and every $\alpha\in \bN$, there exists a $\beta\in \bN$ such
that
\[
  \sym^{[\alpha \bdot k]}\sF \otimes \mathcal H^{k}
\]
is globally generated over $U$, for any multiple $k$ of $\beta$.

By \cite{Kaw81} and \cite{Viehweg83} (see also \cite{Fuj78}, \cite{Zuc82} and
\cite{Kol86}) it is known that for a projective morphism $f: X\to B$ of regular
quasi-projective varieties $X$ and $B$, with connected fibers, $f_*\omega^m_{X/B}$ is
weakly positive over $B \backslash \disc(f)$, for any $m\in \bN$.  One can then
slightly generalize this to the case of mildly singular families as follows.  For any
projective morphism $f:X\to B$ of quasi-projective varieties $X$ and $B$, if $B$ is
nonsingular and $X$ has only canonical singularities, then
the torsion free sheaf $f_*\omega^{[m]}_{X/B}$ is weakly positive for every $m\in\bN$
for which $\omega^{[m]}_{X/B}$ is invertible (see also~\cite{Fuj18}).

\subsection{Singularities in linear systems of ample line bundles}
\label{subsection:EV}

\begin{def-not}\label{def:EV}
  Following the definition of Esnault-Viehweg \cite{EV92}*{Def.~7.4} for a line bundle $\sL$ on projective manifold $X$, 
  with $H^0(X, \sL) \neq 0$, we
  define
  \[
    e(\sL): = \sup\left\{ \left\lfloor \frac{1}{\rm{lct}(\Gamma)} \right\rfloor +1 \;
      \bigg| \; \Gamma\in |\sL| \right\} ,
  \]
  where $\rm{lct}(\sblank)$ denotes the log-canonical threshold.
\end{def-not}
\hmarginpar{\tiny{B: I had to change this definition a little to be absolutely
    consistent with [EV]. I think $\frac{1}{\rm{lct}}$} would also work but then a
  reviewer may complain that since the definition is not entirely consistent with
  [EV] we have to reprove a few things that we need later on in Section 3.}%

\noindent
If $\sL$ is very ample, then 
\begin{equation}
  \label{eq:6}
  e(\sL^{m}) \leq m \bdot c_1(\sL)^{\dim X}+1 ,
\end{equation}
for every integer $m>0$ by \cite{EV92}*{Lem.~7.7}.  Therefore, for the set of very
ample line bundles on projective manifolds with fixed Hilbert polynomial $h$, the
number $e(\sblank)$ has an upperbound depending on $h$.  \hmarginpar{\tiny S:
  and $\nu$, no?  \textbf{B: Well, if you take powers then yes $\nu$ as well but I
    found it unnecessary to mention that here but of course in what comes up next it
    shows up. This sentence was supposed to be segue.}}%
As the moduli functor of canonically polarized manifolds with fixed canonical
\hmarginpar{\tiny S: this is not used consistently later. I think we should probably
  just say somewhere that by ``Hilbert poly'' we will mean the Hilbert poly of the
  canonical bundle. \bf{B: Done. See Convention below.}}%
Hilbert polynomial $h$ is bounded \cite{Mat72} (see \cite[Def.~1.15 (1)]{Viehweg95}
for the definition), one can find an integer
\begin{equation}\label{a_0}
  a_0=a_0(h)\in \bN
\end{equation}
such that, for every integer multiple $m$ of $a_0$, the line bundle $\omega^m_X$ is
\hmarginpar{\tiny S: changed ``very ample'' to ``strongly ample''}%
strongly ample (cf.~\autoref{def:strongly-ample}), and then by \autoref{eq:6} there
exists an $e_m\leteq e_m(h)\in \bN$, depending only on $m$ and $h$, that satisfies
the inequality
\begin{equation}\label{eq:EVIneq}
  e(\omega^m_X) \leq e_m 
\end{equation}
for every manifold as above.

As we mentioned in \autoref{rem:volume-instead-of-Hilb-poly}, we will be following the
terminology of \cite{ModBook} with regard to moduli functors and moduli spaces. In
particular, instead of a Hilbert polynomial we will be using the dimension, $n$, and the
canonical volume $\nu$. So, we may rephrase the above statement using $(n,\nu)$ in
place of $h$, and say that $a_0=a_0(n,\nu)$ and $e_m=e_m(n,\nu)$ depend on these
quantities.

\begin{notation}\label{not:e}
  For every projective morphism of normal schemes $f:X \to B$ whose general fiber is
  a canonically polarized manifold, every sufficiently divisible $m\in \bN$ and any
  $a\in \bN$, we set $\tma{a}: = m^a \bdot e_m \bdot r_m$, where $e_m$ is as in
  \autoref{eq:EVIneq} and $r_m:=\rank( f_*\omega^{[m]}_{X/B})$.
\end{notation}

\hmarginpar{\tiny S: I removed the ``convention'' for the Hilbert polynomial as we
  probably don't need it anymore.}

\section{\Aug direct image sheaves and Viehweg numbers}
\label{sect:Section3-Prelim}

\subsection{Generic global generation of \aug direct image sheaves in stable
  families} 
We define a notion of \aug direct images sheaves, which is closely connected to our
notion of Viehweg numbers, to be introduced later in this section.  \hmarginpar{\tiny
  S: I changed ``augmented'' to ``twisted''. I think augmented means something
  completely different. But we can change it easily as I replaced the word with a
  command. \bf{B:Twisted sounds good. I did not use twisted because some -especially
    analytic- people use this with some positivity assumptions on the twist which has
    nothing to do with what we do.}}

\begin{definition}[\Aug direct image sheaves]\label{def:Augmented}
  In the settings of \autoref{not:L}, \autoref{not:e}, set $V:= B\backslash D$ and
  assume that for all $v\in V$ each fiber $X_v$ is regular of dimension $n$ and of
  canonical volume $\nu$.
  For each multiple $m\in \bN$ of $a_0(n,\nu)$ (cf.~\autoref{a_0}), we define the
  \emph{\aug direct image sheaf} by the following (recall, that
  $t_{m,2}=m^2e_mr_m$):%
  \hmarginpar{\tiny S: I don't like $f^{\{t\}}$!!!! \\
    B: In the non-stable we, no need for reflexive powers. I deleted the brackets.}%
  \[
    \sK^{\nu}_m(f) : =
    \begin{cases}
      f^{\tma 2}_* \omega^{[m]}_{X^{\tma 2}/B} \otimes \Lm(f)^{-m} & \text{if
        $f$ is stable, and } \\
      f^{(\tma 2)}_* \omega^{m}_{X^{(\tma 2)}/B} \otimes \Lm(f)^{-m} & \text{if $f$
        is not stable, but $X$ is regular.}
    \end{cases}
  \]
  To avoid cumbersome notation, when there is no ambiguity, we omit $f$ from $\Lm(f)$
  in the notation.  \hmarginpar{\tiny S: we should probably decide whether we want to
    say ``regular'' or ``smooth''. Here $B$ is declared regular, while $X_v$ is
    smooth. We should be consistent. \bf{B: Let's go with regular.}}
\end{definition}
\hmarginpar{\tiny S: I am assuming by $n_0$ you mean the index, but this is not
  defined anywhere before page 15. It is also not something one can expect the reader
  to remember...so I changed it to be more concrete. \\ Also, directing the reader
  for the notation to an entire section is not really fair...  \bf{B: Thanks. I have
    also labelled an independent equation where we define $a_0$ and referred to it in
    Def. 3.1.}}

\subsection{The stable case}

\begin{proposition}[Base change for $\sK^{\nu}_m$]\label{prop:BC}
  Let $f: X \to B$ be a stable family with a normal general fiber. Assume that $B$ is
  a regular quasi-projective variety
  and $\psi: B'\to B$ is a morphism of finite type, with $B'$ also regular. Let $Z$
  be the main component of $X_{B'}$,\hmarginpar{\tiny S: this used to be $Y_{B'}$. I
    assumed it was a typo}%
  equipped with the natural projection $g: Z\to B'$.  Finally, let $m\in\bN$ be such
  that $\omega^{[m]}_{X/B}$ is a strongly $f$-ample line bundle
  (cf.~\autoref{def:strongly-ample}).  Then
  $\psi^* \sK^{\nu}_m(f) \cong \sK^{\nu}_m(g)$.
\end{proposition}

\begin{proof}
This directly follows from \autoref{rem:BaseChange}. That is, we consider the two 
morphisms 
\[
f^{\tma 2} :  X^{\tma 2} \to B   \;\;\;  , \;\;\;  g^{\tma 2}: Z^{\tma 2} \to B' ,
\]
and observe that by \autoref{prop:rfold} and \autoref{rem:BaseChange} we have 
\[
\psi^* f_*^{\tma 2} \omega^{[m]}_{X^{\tma 2}/B}  \cong  g^{\tma 2}_* \omega^{[m]}_{Z^{\tma 2}/ B'}  .
\]
On the other hand, if $\omega^{[m]}_{X/B}$ is strongly $f$-ample, then 
$\psi^* \Lm(f)\cong \Lm(g)$, which gives the desired isomorphism.
\end{proof}

\begin{proposition}\label{prop:wp}
  Let $f:X\to B$ be a stable family, where $B$ is a regular quasi-projective variety.
  Assume that there is an open subset $V\subseteq B$ such that for every $v\in V$,
  the fiber $X_v$ is regular of dimension $n$ and canonical volume $\nu$.
  Then $\sK^{\nu}_m(f)$ is weakly positive for every multiple $m$ of $a_0(n,\nu)$
  (cf.~\autoref{a_0}).
\end{proposition}

\begin{proof}
  Note that as $B$ is regular, it follows that
  $X^{\tma 2}_V=\left(f^{\tma 2}\right)^{-1}V\subseteq X^{\tma 2}$ is regular.  Let
  $\mu: X^{(\tma 2)} \to X^{\tma 2}$ be a resolution which is an isomorphism over $V$
  and let $f^{(\tma 2)}=f^{\tma 2}\circ\mu: X^{(\tma 2)} \to B$ denote the induced
  family. %
  \hide{Because $f^{\tma 2}: X^{\tma 2}\to B$ is stable,
    the fibers 
    of $f^{\tma 2}$ are Du Bois by \cite{MR2629988}*{Thm.~1.4} and therefore
    $X^{\tma 2}$ has rational 
    singularities by \cite{KS13}*{Thm.~7.1}.
  }%
  Next, let $E$ be an exceptional divisor on $X^{(\tma 2)}$ such that
  $K_{X^{(\tma 2)}} +E \sim_{\bQ} \mu^* K_{X^{\tma 2}}$.

  Recall that by \autoref{prop:inj-for-stable} (cf.~\autoref{lem:Lm-contains}) there
  is an injection
  \begin{equation}\label{SplittingInjection}
    \Lm^{m^2\bdot e_m}   \hooklongrightarrow  f_*^{\tma 2} \omega^{[m]}_{X^{\tma
        2}/B}. 
  \end{equation}
  Let $\Gamma\subset X^{\tma 2}$ be the effective Cartier divisor corresponding to
  the global section of
  $\omega^{[m]}_{X^{\tma 2}/B} \otimes (f^{\tma 2})^* \Lm^{- m^2\bdot e_m}$ induced
  by the adjoint morphism of \autoref{SplittingInjection}. In particular we have 
  \[
    \sO_{X^{\tma 2}} (\Gamma) \simeq \omega^{[m]}_{X^{\tma 2}/B} \otimes (f^{\tma 2})^*
    \Lm^{- m^2\bdot e_m}.
  \]
  Setting $\wtilde \Gamma:= \mu^* \Gamma$ leads to:
  \[
    \sO_{X^{(\tma 2)}} (\wtilde \Gamma)\simeq \mu^*\omega^{[m]}_{X^{\tma 2}/B}
    \otimes \left( f^{(\tma 2)} \right)^* \Lm^{- m^2\bdot e_m} \simeq
    \omega_{X^{(\tma 2)}/B}^m (mE) \otimes \left( f^{(\tma 2)} \right)^*
    \Lm^{- m^2\bdot e_m} .
  \]
  Next, we define the invertible sheaf $\wtilde \sM$ on $X^{(\tma 2)}$ by
  \[
    \wtilde \sM : = 
    \left( \omega_{X^{(\tma 2)}/B} (\rup{E}) \right)^{m-1} \otimes \left(
    f^{(\tma 2)} \right)^* \Lm^{-m} .
  \]
  We thus have %
  \hmarginpar{\tiny S: I think that \\
    1) there is no need for the $\wtilde{\ }$ on $X^{\tma 2}$ in the first line and \\
    2) the coefficient of $\rup{E}-E$ at the end was incorrect 
    here. Please check.}%
  \begin{equation}
    \label{StarPull}
    \begin{multlined}
      \wtilde \sM^{m\bdot e_m} (- \wtilde \Gamma) \simeq \omega_{
        X^{(\tma 2)}/B}^{m \left( e_m(m-1) -1 \right)}
      \left(me_m(m-1)\rup{E}-mE\right) \simeq \\
      \simeq \mu^* \omega_{X^{\tma 2}/B}^{ \left[ m \left( e_m(m-1) -1 \right)
        \right] }\left(me_m(m-1)(\rup{E}-E)\right).
    \end{multlined}
  \end{equation}
  Notice that $\rup{E}-E$ is an effective (exceptional) divisor and hence
  \begin{equation}
    \label{eq:2}
    \mu^* \omega_{X^{\tma 2}/B}^{ \left[ m \left( e_m(m-1) -1 \right) \right]}\subseteq
    \wtilde \sM^{m\bdot e_m} (- \wtilde \Gamma).
  \end{equation}
  Recall that $\omega^m_{X^{\tma 2}_v}$ is strongly ample 
  for every $v\in V$ by the choice of $a_0$.
  \hmarginpar{\tiny S: Since we switched to ``strongyl ample'', the Grauert argument
    is not needed here, so I commented it out.}%
  \hmarginpar{\tiny S: we don't need Siu here... {\bf B: Right! Thanks. But why/how
      are we using Grauert?  with the combination of IPG and GG on fibers we have
      relative GG = surjective of the next natural morphism. ALSO I introduced $m'$
      to cover all multiples of $m$ and not just $m$.} S: Good, thanks! \\ S:
    Grauert's Thm implies that the fibers (=stalks tensored with the residue field)
    of the pushforward are isom to the global sections of the original sheaf
    restricted to the corresponding fiber of $f$. But the restriction of this power
    of the relative dualizing sheaf to a fiber is very ample and hence generated by
    global sections. This implies that the displayed morphism is surjective when we
    restrict it to a fiber and then it is surjective by Nakayama. \bf{B: Yes of
      course. I forgot to make a reference to this theorem. I added one
      now. Unfortunately I tend to refer to this as Cohomology Base Change (in my
      mind this is Cohomology Base Change I).}}%
  Then, it follows from Grauert's theorem \cite[Cor.~12.9]{Ha77} that
  the natural morphism
  \[
    \left( f^{\tma 2} \right)^* \underbrace{ f^{\tma 2}_* \omega_{X^{\tma 2}/B}^{ \left[
        m \left( e_m(m-1) -1 \right) \right]} }_{:= \sF} \longrightarrow \omega_{X^{\tma
        2}/B}^{ \left[ m \left( e_m(m-1) -1 \right) \right] }
  \]
  is surjective over $V$.  Note that by the
  choice of $\mu$, the $\mu$-exceptional divisor $E$, and hence $\rup{E}-E$ is
  disjoint from $X_V$, which implies that after pulling back by $\mu$ and using
  \autoref{StarPull}, the obtained morphism
  \[
    \xymatrix{%
      \left( f^{(\tma 2)} \right)^* \sF \ar[r] \ar@/^2em/[rr]& \mu^*\omega_{X^{\tma
          2}/B}^{ \left[ m \left( e_m(m-1) -1 \right) \right] } \ \ar@{^(->}[r] &
      \wtilde \sM^{m\bdot e_m} ( - \wtilde \Gamma) , }
  \]
  is surjective onto $\wtilde \sM^{m\bdot e_m} (- \wtilde \Gamma)$ over $V$ and hence
  generically surjective over $B$ (cf.~\autoref{eq:2}).

  On the other hand, for every $v\in V$, we have
  \begin{align*}
    m\bdot e_m \geq   e_m
    & \geq   e( \omega_{X_v}^{[m]} ) \\
    & =   e(\omega_{X_v^{\tma 2}}^{[m]}) \;\;\; \text{by \cite{Viehweg95}*{Cor.~5.21}  }  \\
    & \geq  \left\lfloor \frac{1}{\mathrm{lct}(\Gamma|_{X_v^{\tma 2}})} \right\rfloor +1 =  
      \left\lfloor  \frac{1}{\mathrm{lct}(\wtilde \Gamma|_{\wtilde X_v^{\tma 2}})}
      \right\rfloor +1 . 
  \end{align*}
  Now, as $\sF$ is weakly positive, by vanishing results due to Koll\'ar and
  Kawamata-Viehweg cf.~\cite{Viehweg95}*{\S2.4} and more precisely by
  \cite{Vie-Zuo03a}*{Prop.~3.3}, we find that
  \begin{equation}
    \label{WPSheaf}
    \left( f^{(\tma 2)} \right)_*  \left(   \omega_{X^{(\tma 2)}/B} \otimes
    \wtilde \sM \right) 
  \end{equation}
  is weakly positive.

  \begin{claim}
    \label{claim:inclusion}
    Let $f_U: U\to V$ denote the restriction of $f:X\to B$ to $V$.  There is a
    natural injection
    \[
      \mu_* \left( \omega_{X^{(\tma 2)}/B} \otimes \wtilde \sM \right)
      \hooklongrightarrow \omega^{[m]}_{X^{\tma 2}/B} \otimes \left( f^{\tma 2} \right)^*
      \Lm^{-m} ,
    \]
    that is an isomorphism over $U^{\tma 2}$.
  \end{claim}

  \begin{proof}[Proof of \autoref{claim:inclusion}]
    By the definition of $\wtilde \sM$ we have \hmarginpar{\tiny S: added ( ) to
      $X^{\tma 2}$ in the first line}%
    \begin{equation}
      \begin{multlined}
        \omega_{X^{(\tma 2)}/B} \otimes \wtilde \sM = \omega_{
          X^{(\tma 2)}/B} \otimes
        \left(\omega_{X^{(\tma 2)}/B} (\rup{E}) \right)^{m-1}
        \otimes
        \left( f^{(\tma 2)} \right)^* \Lm^{-m} \simeq  \\
        \simeq \omega_{X^{(\tma 2)}/B}^m\left( (m-1)\rup E\right) \otimes \left(
          f^{(\tma 2)} \right)^* \Lm^{-m}  \simeq \\
        \simeq \mu^* \left( \omega^{[m]}_{X^{\tma 2}/B} \otimes (f^{\tma 2})^*
          \Lm^{-m} \right)\otimes \sO_{X^{(\tma 2)}/B}\left( (m-1)\rup
          E-mE\right)\subseteq \\
        \subseteq \mu^* \left( \omega^{[m]}_{X^{\tma 2}/B} \otimes (f^{\tma 2})^*
          \Lm^{-m} \right)\otimes \sO_{X^{(\tma 2)}/B}\left( m(\rup E-E)\right) ,
      \end{multlined}
    \end{equation}
    from which the required injection follows, because $\rup E-E$ is an effective
    $\mu$-exceptional divisor and hence
    $\mu_*\sO_{X^{(\tma 2)}/B}\left( m(\rup E-E)\right) \simeq \sO_{X^{\tma 2}/B}$,
    cf.~\cite{KMM87}*{Lem.~1-3-2}.
  \end{proof}

  \noin Now, using \autoref{claim:inclusion} and the fact that \autoref{WPSheaf} is
  weakly positive, it follows that so is
  \[
    f^{\tma 2}_* \left( \omega^{[m]}_{X^{\tma 2}/B} \otimes (f^{\tma 2})^* \Lm^{-m}
    \right) \simeq \sK^{\nu}_m(f) .\qedhere
  \]
\end{proof}


In the situation of \autoref{prop:wp}, let $\mathcal H$ be any ample line bundle on
$B$.  Then, for every $m$ as in \autoref{prop:wp}, 
there is a $\beta_m\in \bN$ such that, for every multiple $k$ of $\beta_m$,
\begin{equation}
  \label{cor:prelim}
  \mathrm{Sym}^{k}(\sK^{\nu}_m(f) ) \otimes \mathcal H^{m\bdot k} \; \;   \text{is 
    generically generated by global sections.} 
\end{equation}

\begin{remark}
  \label{rk:wp}
  \hmarginpar{\tiny S: Since this is important later, I think this should be flashed
    out. Saying things like ``one can prove'' is a red flag for referees. \\
    B: I added some more details. This is the easy version of \autoref{prop:wp} and I
    could not think of an efficient way of reproducing the proof. Hopefully with
    these replacements that I have added in the remark it's straightforward to follow
    how each step would be identical but much easier. \\ \color{blue} S: OK, good.}%
  The conclusion of \autoref{prop:wp} also holds for any snc morphism
  $f: (X, \Delta_f) \to (B, D_f)$, with quasi-projective $B$ and canonically
  polarized regular fibers.  More precisely, using \autoref{prop:injection}, taking
  $\sO_{X^{(t_{m,2})}} = \omega^m_{X^{(t_{m,2})}}\otimes (f^{(t_{m,2})})^*\sW_m^{-m^2
    e_m}$ and replacing $\wtilde \sM$ by
  $$
  \sM = \omega^{m-1}_{X^{(t_{m,2})}/B} \otimes \Big( f^{(t_{m,2})} \Big)^*   \sW^{-m}_m  ,  
  $$ 
  from the proof 
  of \autoref{prop:wp} it follows that $\sK^{\nu}_m(f)$ is weakly positive. 
\end{remark}

\subsection{The \reduced case} 

\begin{lemma}[\Aug direct image sheaves under semistable reductions]
  \label{lem:SSBaseChange}
  In the setting of \autoref{lem:injection}, assume that $X\setminus \Delta$ is
  regular.  Let $\wtilde X\to X$ be a strong resolution such that
  $\wtilde f: (\wtilde X, \Delta_{\wtilde f}) \to (B,D_{\wtilde f})$ is an snc
  morphism and $g: Y \to B'$ a \reduced morphism as in \autoref{TWO}.  Then, there is
  an injection
 \[
    \sK^{\nu}_m(g) =    g^{ (\tma 2) }_*  \omega^m_{ Y^{(\tma 2)}/ B' }
    \otimes \Lm^{-m} (g)  
    \hooklongrightarrow   \gamma^* \sK_m^{\nu} (\wtilde f)  ,
  \]  
  naturally defined by a generic isomorphism over $B'$.
\end{lemma}

\begin{proof}
  By \autoref{TWO} 
  there is an injection
   \begin{equation}
    \label{1stInject}
    \Lm(g)^{-1} \hooklongrightarrow \gamma^* \Lm(\wtilde
    f)^{-1}.
  \end{equation}
  For a suitable choice of strong resolutions 
  $\wtilde X^{(\tma 2)}$ and $Y^{(\tma 2)}$ there is a commutative diagram
  \[
    \xymatrix{ Y^{(\tma 2)} \ar[rr] \ar[d]_{g^{(\tma 2)}} && \wtilde X^{(\tma 2)}
      \ar[d]^{\wtilde f^{(\tma 2)}}  \\
      B' \ar[rr]^{\gamma} && B . }
  \]
  By \cite[\S 3, p.~336]{Viehweg83}
  \hmarginpar{\tiny S: Can we make this reference more explicit?}%
  we have
  \begin{equation}\label{VieMor}
    g^{(\tma 2)}_*  \omega^m_{Y^{(\tma 2)}/B'}  \subseteq  \gamma^* \wtilde
    f^{(\tma 2)}_* \omega^m_{\wtilde X^{(\tma 2)}/B} . 
  \end{equation}
  This inclusion and \autoref{1stInject}, raised to the power $m$, gives the required
  injection.
\end{proof}

\subsection{Viehweg numbers}
\label{subsect:Vnumbers}

\begin{definition}\label{def:M}
  In the setting of \autoref{def:Augmented}, assuming that $B$ is quasi-projective,
  let $\cH$ be any ample line bundle on $B$.  We define the line bundle $\sM$ as
  follows.
  \begin{enumerate}
  \item\label{item:SG} If $f$ is stable and $\omega_{X^{\tma 2}/B}$ is a line bundle,
    e.g., if $f$ is Gorenstein, 
    \hmarginpar{\tiny S: Do we really need Gorenstein, or just that $\omega$ is a
      line
      bundle? \\
      B: just that $\omega$ is a line bundle? \\ \color{blue} S: are you asking or
      saying? If saying, then we should
      replace ``Gorenstein'' with that.  \\
      \color{red}: B: I am SAYING. Feel free to change Gorenstein to relative
      $\omega$ being invertible.}%
    then
    \[
      \sM: = \omega_{X^{\tma 2}/B} \otimes (f^{\tma 2 })^* (\Lm^{-1}\otimes \cH),
    \]
    and
  \item\label{item:SMOOTH} if $X$ is regular, but $f$ is not stable, then 
    \[
      \sM: = \omega_{X^{(\tma 2)}/B} \otimes (f^{(\tma 2)})^* (\Lm^{-1}\otimes \cH),
    \]
  \end{enumerate}
  %
    %
  \hmarginpar{\tiny S: I am confused by this assumption. If we assume this for
    $\beta=1$, doesn't that imply for all $\beta\in\bN$? OK, I changed it, but the
    original paragraph is here commented out. Please compare. \\
    B: Not sure how this was changed but this is what I meant.}%
  Now, for some $\beta\in \bN$, assume that $H^0(\sM^{m\beta}) \neq 0$. Fix a
  non-zero section $0\neq s_m\in H^0(\sM^{m\bdot \beta})$.  \hmarginpar{\tiny S: what
    is $\beta_m$?  Is this perhaps a typo from copy-pasting from \autoref{cor:KEY}? I
    changed it to
    $\beta$. Please check. \\
    B: There was a typo. There should not be a $\beta_m$ in this subsection except
    for right after \autoref{def:V}}%
  Let $\sigma_{s_m}: Z'_{s_m}\to X^{\tma 2}$, respectively
  $\sigma_{s_m}: Z'_{s_m}\to X^{(\tma 2)}$, be the cyclic covering associated to
  $s_m$ cf.~\cite{Laz04-I}*{Prop.~4.1.6}, for $\sM$ as in \autoref{item:SG} and
  \autoref{item:SMOOTH}.
  \hmarginpar{\tiny S: Do we really need \cite{BG71} here? On the next page we're
    referring to \cite{Laz04-I}*{Prop.~4.1.6} for this, so why not here, too?}%

  Let $\mu: Z_{s_m}\to Z'_{s_m}$ be a resolution
  with the induced family $g_{s_m}: Z_{s_m}\to B$ and the commutative diagram
  \begin{equation}\label{eq:Cyclic}
    \begin{gathered}
      \xymatrix{ Z_{s_m} \ar[drrrr]_{g_{s_m}} \ar[rr]^{\mu} && Z'_{s_m}
        \ar[rr]^{\sigma_{s_m}}
        &&  X^{\tma 2} \text{ or }  X^{(\tma 2)} \ar[d]^f\hskip-.7em \\
        &&&& B .  }
    \end{gathered}
  \end{equation}
\end{definition}

\begin{definition}
  \label{def:V}
  With a fixed ample line bundle $\mathcal H$ on $B$, and using the notation of
  \autoref{def:M}, we define the \emph{Viehweg
    number} of $f$ associated to the triple $(m, \beta, s_m )$ to be the rank of the
  following local system \hmarginpar{\tiny S: should $\cH$ be also in the index? I
    know it is a lot, but it is already a lot without that. Perhaps say that we fix
    $\cH$? Or do this?\\ And then later we could omit $m$ and $\beta$ from the
    notation... What do you think? (I also removed the parantheses from the index as
    later the $\alpha$
    appears without them...)\\
    B: I added $\mathcal H$ in the first line. For the notation I think you have also
    introduced $\alpha_{s_m}(f)$ just here in this definition. I think it's nice and
    compact. I am happy if you want to use it for the rest of the paper. I did not
    want to make a global change before being sure about your intentions.\\
    \color{blue} S: I had similar thoughts. I added the compact notation, so we could
    change it to that. If you'd like and willing, go ahead and change it
    everywhere. (We can also add a sentence saying something like ``In order to keep
    the notation manageable, we will suppress some of the parameters in the sequel,
    but will remember that $\alpha$ depends on those as well.''}%
  \[
    \alpha_{s_m}(f)=\alpha_{m, \beta, s_m} (f): = \rank \left( \R^{\dim (X^{ \tma
          2 }/B)} (g_{s_m})_* \bC_{(Z_{s_m}\backslash \Delta_{g_{s_m}})} \right) .
  \]
\end{definition}

In the setting of \autoref{def:M} 
assume in addition that the smooth and Gorenstein fibers are canonically polarized.
Then \autoref{prop:wp}, \autoref{cor:prelim} and \autoref{rk:wp} imply that, after
removing a subset of $B$ of $\codim_B\geq 2$, if necessary, there exist $m$, and
$\beta_m \in \bN$ such that $H^0(\sM^{m\bdot \beta_m})\neq 0$ (on the respective
varieties in the stable and the regular case).  Therefore \autoref{def:V} is relevant
for all such families.

\begin{notation}\label{not:VNumber}
  In order to keep the notation manageable, we will suppress some of the parameters
  in \autoref{def:V} and use the notation $\alpha_{s_m}(f)$ instead of
  $ \alpha_{m, \beta,\cH, s_m} (f)$, but will keep in mind the choices that we have
  made to define $s_m$.
\end{notation}

\begin{remark}
  \label{rk:restrict}
  Let $f_{U'}: U'\to V'$ be a subfamily of $f:X\to B$, that is, $ f_{U'}=f|_{U'}$.
  Then, we clearly have (with the same $m$ and $\beta$),
  \[
   \alpha_{s_m} (f) = \alpha_{s'_m} (f_{U'})
  \]
  where $s_m |_{U'} = s'_m$.  \hmarginpar{\tiny S: Do we really need the
    $(\ )_0$ here? \\ Why not $ s_m |_{U'} = s'_m $? \\
    \color{red} B: Done. }
\end{remark}

\begin{proposition}
  \label{prop:IsomCover}
  Let $\sL_1$ and $\sL_2$ be two line bundles on a variety $U$
  and assume that there is an isomorphism $\eta: \sL_1\to \sL_2$. Let
  $0\neq s_{m,2}\in H^0(U, \sL_2^m)$, for some $m\in \bN$, and set
  $s_{m,1}:= \eta^* s_{m,2}\in H^0(U, \sL_1^m)$. Let $\sigma_{s_{m,1}}:Z'_1\to U$ and
  $\sigma_{s_{m,2}}:Z'_2\to U$ be the cyclic coverings associated to $s_{m,1}$ and
  $s_{m,2}$. Then, there is a natural isomorphism $Z'_1 \simeq Z'_2$, induced by
  $\eta$, that commutes with $\sigma_{s_{m,1}}$ and $\sigma_{s_{m,2}}$.
  \end{proposition}

\begin{proof}
  This directly follows from the construction of $Z'_i$
  cf.~\cite{Laz04-I}*{Prop.~4.1.6}.
  More precisely, let
  \[\mathbf L_i: = \Spec_{\sO_U}(\sym \sL_i^*),\]
  with the natural projection $p_i: \mathbf L_i \to U$, and set
  $T_i\in H^0(\mathbf L_i, p_i^*\sL_i)$ to be the global section associated to the
  tautological map $\sO_{\mathbf L_i} \to p_i^*\sL_i$.  Let \hmarginpar{\tiny S:
    doesn't that map ($\overline\eta$) go the other way? \\
    B:  Since this is an isomorphism it really doesn't matter. Does it?\\
    \color{blue} S: since it says ``induced by $\eta$'', it does. :) \\
    \color{red} B: I see. Is it OK now? \\ \color{green} S: Yes, thanks!}%
  $\overline \eta: \mathbf L_2 \to \mathbf L_1$ denote the natural isomorphism
  induced by $\eta$.  By construction we have $Z'_i= (T_i^m- p_i^*s_i)_0$, and that
  $\overline{\eta}^*(T_1^m - p_1^*s_{m,1}) = T_2^m- p_2^* s_{m,2}$, inducing the
  desired isomorphism.
\end{proof}

\noin The next lemma now follows by combining \autoref{prop:IsomCover} with
\autoref{rk:restrict}.

\begin{lemma}
  \label{lem:VNIso}
  Let $f:X \to B$ and $\sM$ be as in \autoref{def:M} and let $V\subseteq B$ an open
  set and $f_U:U=X_V\to V$ the corresponding subfamily.  Consider a line bundle $\sN$
  on $U$ that is equipped with an isomorphism
  $\eta: \sN \overset{\simeq}{\longrightarrow} \sM|_U$.  
  Let $s_{\!\sN}:= \eta^*s_m|_U$, $Z_{\eta}\to U$ a resolution of
  singularities 
  of the cyclic covering associated to $s_{\!\sN}$, and $g_{\sN}: Z_{\sN}\to V$ the
  induced morphism. Then, 
  \[
    \rank\left( \R^{\dim (X^{\tma 2}/B)}(g_{\sN})_* \bC_{Z_{\sN} \backslash
        \Delta_{g_{\sN}} } \right) = \alpha_{s_m}(f) . 
  \]
\end{lemma}

\section{Finite-type substacks of the stack of canonically polarized manifolds and
  boundedness of Viehweg numbers}\label{sect:Section4-invariant}

In this section we first recall the main results of \cite{KoL10} and
\cite{Bedulev-Viehweg00} regarding parametrizing spaces for canonically polarized
families. As we have already mentioned in \autoref{rem:volume-instead-of-Hilb-poly}
we will follow the terminology of \cite{ModBook} and use the dimension and the
canonical volume instead of the Hilbert polynomial. The above papers used Hilbert
polynomials, but by \cite{ModBook}*{5.1,6.19} the two approaches are equivalent.

Subsequently, we will use these to establish certain uniform global generation
results for the twisted direct image sheaves $\sK_m^{\nu}(f)$ of canonically polarized
families introduced in \autoref{def:Augmented}. We start by reviewing numerical
bounds arising from Arakelov inequalities over curves.

\hmarginpar{\tiny S: why not $M^\circ$ and $M$ as for $C$ in
  \autoref{def:weak-bound}? \\
  B: Changed now. Please have a look.\\ \color{blue} S: Hm. Did you mean the first
  $M$ to be $M^\circ$ and the last $M^\circ$ to be an $M$? \\
  \color{red} B: Yes. Sorry.\\ \color{green} Do you really only want $\sL$ be ample
  on $M^\circ$? }%
Let $M^\circ$ be a connected and finite type scheme (not necessarily irreducible)
with a \emph{compactification} $M^\circ \subseteq M$ as a subscheme of a projective
scheme $M$, equipped with a fixed line bundle $\sL$, which is ample on
$M^\circ$. Depending on $\sL$, let $b_{\sL}: \bZ^2_{\geq 0}\to \bZ_{\geq 0}$ be a
function in two variables.

\begin{definition}[Weak bound, \cite{KoL10}*{Def.~2.4}]\label{def:weak-bound}
  \hmarginpar{\tiny S: this is not really how we defined it we said that
    $\mu_{C^\circ}$ is \emph{weakly bounded with respect to $\sL$} if such a $b_\sL$
    exists and then that it is\emph{weakly bounded by $b_\sL$}. The phrase
    ``\emph{$b_{\sL}$-weakly bounded}'' sounds awkward to me, but it could be just
    me... \\
    B: I edited this definition. The definition is the same as yours though. If the
    inequality like this holds, then $b_{\sL}$ is assumed to exist by
    default. Writing it this way was my way of shortening your somewhat long set up
    for this definition. \\ \color{blue} S: Yeah, I realized this afterwards. When I
    wrote this comment, I had only seen this much. Reading on I realized that it was
    OK as you wrote, but I forgot to remove my comment... }%
  
  Let $V$ be a regular quasi-projective variety (of arbitrary dimension). Then a
  morphism $\mu_V: V\to M^\circ$ is called \emph{weakly bounded with respect to
    $b_{\sL}$}, if any regular curve $C^\circ$ equipped with a morphism
  $i: C^\circ \to V$ satisfies the following property: Let $C$ be the regular
  compactification of $C^\circ$ of genus $g$ and $d:= \deg (C\setminus
  C^\circ)$. Then the extension $\mu_C: C\to M$ of the naturally induced map
  $\mu_{C^\circ}: C^\circ \to M^\circ$ satisfies the inequality
  \[
    \deg ( \mu_C^*\sL ) \leq b_{\sL} (g,d) .
  \]
\end{definition}

\begin{notation} 
  Let $n,\nu\in\bN$. Recall that $\smn(V)$ denotes the class of smooth, projective
  and canonically polarized families $f_U:U\to V$ of varieties of dimension $n$ and
  of canonical volume $\nu=K_{U_t}^n$. 
  \reversemarginpar
  \hmarginpar{\tiny S: why use the Hilber poly in one definition and the volume in
    the
    other?\\ Also, the $d$ in the subscript should probably be $\nu$, right? \\
    B: OK, this is something maybe you can help with. $h$ fixes $\nu$ but not the
    other way around. On the other hand I have seen that this KSB compactifications
    fix $\nu$ and not $h$.  The key here is the existence of (1) $\lambda_m$ and (2)
    $p_m$ for all sufficiently large and divisible $m$ such that
    $\det f'_*\omega^{[m]}_{X'/B'}$ on stable reductions descend -up to the power
    $p_m$- to $M^s_{\nu}$.  So a few questions/comments: (i) the subscript $m$ can be
    dropped but I just wanted to be clear that $\lambda$ depends on some $m$ and is
    not uniques.(ii) I am not sure if [Kol90] guarantees (1) and (2) since even
    [BV00] say they don't know if these things existed at the time. I mean a
    result/progress on the KSB from is surely needed here for an extra ref. (iii) If
    there is no need to differentiate between $h$ and $\nu$ please feel free to
    replace $h$ by $\nu$. \\ \color{blue} S: You are absolutely right that in general
    $h$ fixes $\nu$ but not the other way around. However, Koll\'ar recently proved
    that for stable families it is enough to fix the volume. This is a rather
    surprising result and it is probably better if we don't rely on it. It is safer
    to use $h$. My comment was mainly so we stay consistent. }%
  This is naturally a subclass of $\stn(V)$, the class of stable families of
  varieties of dimension $n$ and of canonical volume $\nu$. 
  This class defines a \emph{good moduli theory} as defined in \cite{ModBook}*{6.10},
  cf.~\cite{ModBook}*{6.16, Thm.~6.18} and admits a coarse moduli space which is projective
  \cite{ModBook}*{Thm.~8.1}. We will denote this coarse moduli space of
  stable $n$-dimensional varieties $X$ with canonical volume $K_X^{n} = \nu$ by
  $M_{n,\nu}$.
\end{notation}

\hmarginpar{\tiny \color{red}\bf S: I wonder if we should make this definition for all
  stable families... That would just mean changing $\smn$ to $\stn$...}%
\begin{definition}[Weakly bounded families]
  A family $(f_U: U\to V) \in \smn(V)$ is called \emph{weakly bounded}, if there
  exists an ample line bundle $\sL$ on $M_{n,\nu}$ and a function $b_{\sL}$ with
  respect to which the induced moduli map $\mu_V: V\to M_{n,\nu}$ is weakly
  bounded. %
  \hmarginpar{\tiny S: same comment as above. Also, in this way there has to be a
    (stated!) connection between $h$ and $\nu$!!}%
\end{definition}

By combining \cite{Kollar90}*{Thm.~2.5},~\cite{Viehweg95}*{Thm.~7.17},
\cite[Thm.~5]{Viehweg10} and \cite[Thm.~1.4]{Bedulev-Viehweg00} we have the following
important fact.

\begin{fact}\label{fact:BV}\hmarginpar{\vskip 1em\tiny S: this originally had $\sL_m$,
    but I am guessing that the ``$m$'' is a typo}%
  There exist an ample line bundle $\sL$ on $M_{n,\nu}$ and a function $b_{\sL}$ for
  which every $f_U\in \smn(V)$ is weakly bounded for any smooth quasi-projective
  variety $V$.
\end{fact}

\subsection{Parameterizing spaces of weakly bounded families}
\label{subsect:Par}

\begin{set-up}\label{SETUP}\hmarginpar{\tiny S: changed ``log-smooth'' to ``snc''}%
  Let $V$ be a smooth quasi-projective variety and $(B,D)$ an snc 
  compactificaiton, that is $B\setminus D \simeq V$, where $D\subset B$ is a reduced
  divisor with simple normal crossing support.  Let $\cH_B$ be an ample line bundle
  on $B$.
\end{set-up}

According to \cite{KoL10}*{Cor.~2.23, Thm.~1.6, Thm.~1.7} with fixed $n$ (dimension)
and $\nu$ (canonical volume) and a suitable choice of a weak bound $b := b_{\sL}$ as
in \autoref{fact:BV}, there is a reduced, connected finite type scheme $W:= W^b$
and a projective family $\sff: \sX \to W\times V=:W_V$ of canonically polarized manifolds 
such that every canonically
polarized family $f_U:U \to V$ of relative dimension $n$ and relative canonical
volume $\nu$
appears in a fiber over
$W$, i.e., for every such $f_U$ there exists a closed point $w\in W$ such that
$\sff_{w}:\sX\resto{\sff^{-1}({\{w\}\times V})}\to \{w\}\times V$ is isomorphic to
$f_U$.
The irreducible components of $\sX$ and the corresponding families will be denoted by
\[
  \{f_{W_{V,i}}: X_{W_{V, i}} \to W_V\}_{1\leq i \leq k}.
\]
Note that $f_{W_V, i}$ is projective and there exists a closed subscheme
$W_i\subseteq W$ such that
\begin{equation}\label{eq:BigFams}
  f_{W_{V,i}} : X_{W_{V,i}} \to W_i\times V 
\end{equation}
is surjective for each $i\in \{1, \ldots, k \}$.

\begin{notation}\label{termin}
  We will write $(f_U: U\to V)\subseteq X_{W_{V,i}}$, if there is a closed point
  $w\in W$ such that $(X_{W_{V,i}})_{\{ w \}\times V} \simeq_V U$.
\end{notation}

\noin By the above we have that
\begin{equation}\label{eq:FP}
  f_U \in \smn(V)  \implies  \exist i, 1\leq i\leq k,  f_U \subseteq
  X_{W_{V,i}}. 
\end{equation}

\begin{remark}[Regularity assumption for $W_i$] %
  Note that one may assume that $W_i$ is smooth.  Indeed, given a resolution
  $\wtilde W_i$ of $W_i$, after replacing $X_{W_V,i}$ by
  \[
    X_{W_{V,i}} \times_{(W_i \times V)} \big( \wtilde W_i \times V \big)
  \]
  and $W_i$ by $\wtilde W_i$, and using the fact that $W_i$ is reduced, we can see
  that the property \autoref{eq:FP} is preserved.
\end{remark}

\begin{notation}
  We will denote $W_i\times V$ by $W_{V,i}$. 
\end{notation}

\medskip

\subsection{Uniform exponents for global generation of twisted direct image
  sheaves}\label{subsect:GGG} 
We will follow the conventions and notation of the previous subsection.

\begin{lemma}\label{lem:gg}
  In the situation of 
  and \ref{subsect:Par}, for every $1\leq i \leq k$, there exist
  \begin{enumerate}
  \item a quasi-projective subscheme $W_i^\circ \subseteq W_i$ with
    $W_{V,i}^\circ:= W_i^0 \times V$, $X_{W^0_{V, i}}:= (X_{W_{V,i}})_{W^0_{V,i}}$,
    $W_{B,i}: = W_i^\circ \times B$, surjective morphism
    $f_{W_{V,i}^0}: X_{W^0_{V,i}} \to W^0_{V,i}$, and divisor
    $D_{W_{B,i}}:= W_i^\circ \times D$,
  \item \label{item:ample} an ample line bundle $\cH$ on $W_{B,i}$ such that
    $\cH|_{ \{ w_i^0 \} \times B } \simeq \cH_B$, for any
    $w_i^0 \in W_i^0$, and
  \item a positive integer $m_0= m_0(n,\nu)$ such that for each multiple $m$ of
    $m_0$, there are integers $\beta^i_m$, $\wtilde \beta^i_m$ and $a_i$, where
    $\beta^i_m = \wtilde \beta^i_m\bdot a_i$, with the following properties:
  \end{enumerate}
  \begin{enumerate}
  \item\label{item:gg} $\omega^m_F$ is globally generated, for every fiber $F$ of
    $f_{W^0_{V, i}}$.
  \item\label{item:any} For every $(f_U: U\to V)\subseteq X_{W^0_{V, i}}$ and any
    smooth compactification $f:X\to B$, the torsion free sheaf
    \[
      \bigotimes^{a_i} \big(\sym^{\wtilde \beta^i_m} \sK^{\nu}_m(f) \big) \otimes
      \cH_B^{m \bdot \beta^i_m}
    \]
    is generically generated by global sections.
  \item\label{BigSheaf}
    $\bigotimes^{a_i}\big( \sym^{\wtilde \beta^i_m} \sK^{\nu}_m(f_{W^0_{V,i}}) \big)
    \otimes \cH^{m \bdot \beta^i_m} \otimes \sO_{W^0_{V,i}}$ is generically
     generated by global sections.
  \end{enumerate}
\end{lemma}

\begin{proof}
  Let $n= \dim X/B$. We may assume that $\dim W_i\neq 0$. Let $W_i^0\subseteq W_i$ be
  an open quasi-projective subset and let $\cH\leteq \pi_B^*\cH_B$,
  where $\pi_B:W_{B,i}\to B$ is the projection. This $\cH$ is an ample line
  bundle on $W_{B,i}$ satisfying \autoref{item:ample}.
  
  Let $f_{W_B,i}: X_{W_B,i}\to W_{B,i}$ denote a proper closure of 
  $f_{W^0_{V,i}}$.
  After removing a subset of $W_{B,i}$ of $\codim_{W_{B,i}} \geq 2$, let
  \[
    \gamma_{W_{B,i}}: W'_{B,i} \longrightarrow W_{B,i}
  \]
  be the cyclic, flat morphism arising from a semistable reduction, cf.~\cite{KKMS}
  and \cite{KM98}*{\S 7.17}.  Set $Y_{W_B,i}:= X_{W_B,i}\times_{W_{B,i}} W'_{B,i}$,
  which is irreducible as $\gamma_{W_{B,i}}$ is flat, with
  $g_{W_{B,i}}: Y_{W_B,i} \to W'_{B,i}$ denoting the natural projection. Let
  $D_{W'_{B,i}}:= \gamma_{W_{B,i}}^{-1}D_{W_{B,i}}$.  

  \smallskip

  For any $f_U\subseteq X_{W_V^0,i}$, let $w\in W^0_i$ be a parametrizing closed
  point as in \autoref{termin}.  Let $B'\subset W_{B',i}$ be the subscheme defined by
  $\gamma_{W_{B,i}}^{-1}(\{w\} \times B)$ and $\gamma: B'\to B$ the induced cyclic
  morphism.  For a generic choice of $w\in W_i^0$, $\gamma$ is flat, ramified only
  along a very ample divisor and $\supp D'$, for some $D'\geq D$, and with the same
  ramification indices as those of $\gamma_{W_{B,i}}$.  Denote the pullback of $f_U$
  via $\gamma$ by $f_{U_{B'}}: U_{B'}\to B' \setminus \gamma^{-1}(D)$.



Now, let $f'_{W_{B,i}}: X'_{W_{B,i}} \to W'_{B,i}$ be the stable reduction of
$f_{W_{B,i}}$ through $\gamma_{W_B,i}$, resulting in the commutative diagram:
\[
  \xymatrix{ X'_{W_{B,i}} \ar[drrr]_{f'_{W_{B,i}}} &&&
    \ar@{-->}[lll]_{\text{birational}} Y_{W_{B,i}} \ar[rrr] \ar[d]^{g_{W_{B,i}}}
    &&& X_{W_{B,i}}  \ar[d]^{f_{W_{B,i}}}   \\
    &&& W'_{B,i} \ar[rrr]^{\gamma_{W_{B,i}}}_{\text{flat and Galois}} &&& W_{B,i} , }
\]
that is there is a resolution $\pi: Z_{W_{B,i}}\to \widehat Y_{W_{B,i}}$ of the
normalization $\widehat Y_{W_{B,i}}$ of $Y_{W_{B,i}}$ such that
$Z_{W_{B,i}} \to W'_{B,i}$ is semistable in codimension one, and that $X'_{W_{B,i}}$
is the canonical model of $Z_{W_{B,i}}$ over $W'_{B,i}$, cf.~Koll\'ar \cite[\S
3]{Kollar10} and Hacon-Xu \cite{HX13}*{Cor.~1.4}. In particular $f'_{W_{B,i}}$ is
stable in codimension one. With no loss of generality, after possibly removing
another subset of codimension two, we may assume that $Z_{W_{B,i}} \to W'_{B,i}$ is
semistable and $f'_{W_{B,i}}$ is stable over $W'_{B,i}$.  Let $n_0$ be the smallest
integer for which $\omega^{[n_0]}_{X_{W'_{B,i}} / W'_{B,i}}$ is invertible and set
$m_0: = \rm{lcm}(m_0, n_0)$, with $a_0$ as in \autoref{a_0}.  Then \autoref{item:gg}
holds for this choice of $m$, and \autoref{prop:wp} is applicable.

According to \autoref{cor:prelim}, for every multiple $m$ of $m_0$, there is an
integer $\wtilde \beta^i_m$ such that, for every multiple $r$ of $\wtilde \beta^i_m$,
%
\begin{enumerate}
\item\label{eq:BigGG} the sheaf
  $\sym^r \sK_m^{\nu}(f'_{W_{B,i}}) \otimes (\cH')^{m \bdot r}$ is globally generated
  over $W'_{B,i}$,
\end{enumerate}
where $\cH':= \gamma^*_{W_{B,i}} \cH$ (recall that we have removed a subset of
$W'_{B,i}$ of $\codim_{W'_{B,i}} \geq 2$).

We next consider the pullback $(X_{W_{B,i}})_{\{w \}\times B}$, which gives a
morphism of reduced pairs $f: (X, \Delta) \to (B,D)$.  Define
$Y:= X \times_{X_{W_{B,i}}} Y_{W_{B,i}}$ with the natural projection $g: Y\to B'$ and
set $X':= (X'_{W_{B,i}})_{B'}$, with the induced family $f':X'\to B'$, which is
stable.  We summarize these constructions in the following commutative diagram.
\begin{equation}\label{eq:BigDiagram}
  \begin{gathered}
    \xymatrix{%
      & X' \ar@{->}@/_6mm/[ddrr]^(0.2){f'} |(.35)\hole |(.60)\hole \ar[dl]
      \ar@{-->}[rr]^{\text{birational}} && Y \ar[dd]^(0.75){g} |(.45)\hole
      |(.50)\hole \ar[rr] \ar[dl] &&
      X \ar[dl]  \ar[dd]^f    \\
      X'_{W_{B,i}} \ar@{->}@/_6mm/[ddrr]^{f'_{W_{B,i}}} \ar@{-->}[rr] &&
      Y_{W_{B,i}}  \ar[rr]  \ar[dd] &&   X_{W_{B,i}} \ar[dd]  \\
      &&& B' \ar[rr]^(0.3){\gamma} |(.44)\hole \ar[dl] && \{w\} \times B
      \ar[dl]  \\
      && W'_{B,i} \ar[rr]^{\gamma_{W_{B,i}}} && W_{B,i} . } 
  \end{gathered}
\end{equation}

Now, we consider the $t^2_m$-fold fiber product of \autoref{eq:BigDiagram} with the
induced stable maps
\[
  \xymatrix{
    (X')^{t^2_m} \ar[rrr]  \ar[d] &&&  (X_{W_{B,i}}')^{t^2_m}   \ar[d]  \\
    X' \ar[rrr]  \ar[d]_{f'} &&&  X_{W_{B,i}}'  \ar[d]^{f'_{W_{B,i}}} \\
    B' \ar[rrr]^i &&& W'_{B,i} , }
\]
where $i: B'\to W'_{B,i} $ denotes the natural inclusion. According to
\autoref{prop:BC}, after removing a subset of $B'$ of $\codim_{B'}\geq 2$, we have
\begin{equation}
\label{claim:pullback}
i^*  \sK^{\nu}_m( f'_{W_{B,i}} )  \simeq \sK^{\nu}_m(f') .
\end{equation}
%




\begin{claim}\label{claim:pullbackGG}
  After replacing $W^0_i$ by an open subset, if necessary, for every multiple $r$ of
  $\wtilde \beta^i_m$, the torsion free sheaf
  $\sym^r \sK^{\nu}_m(f') \otimes \cH_{B'}^{m\bdot r}$ is generically generated 
  by global sections, where $\cH_{B'}:=\gamma^*\cH_B$.
\end{claim}

\begin{proof}[{Proof of \autoref{claim:pullbackGG}}]
  We have already established that
  $\sym^r \sK_m^{\nu}(f'_{W_{B,i}}) \otimes (\cH')^{m \bdot r}$ is globally generated
  over $W'_{B,i}$ cf.~\autoref{eq:BigGG}.  By pulling back and using
  \autoref{claim:pullback} we find that
  \[
    i^* \left( \sym^r \sK^{\nu}_m(f'_{W_{B,i}}) \otimes (\cH')^{m\bdot r} \right)
    \simeq \sym^r \sK^{\nu}_m(f') \otimes \cH_{B'}^{m\bdot r}
  \]
  is generically generated by global sections, for the family $f$ parametrized by a general
  closed point $w\in W^0_i$. 
\end{proof}

From now on we will replace $W_i^0$ by its open subset provided by
\autoref{claim:pullbackGG}.  We note that by construction, for the general family
$f: X\to B$ parametrized by $w\in W_i^0$, and for every codimension one point
$b'\in B'$, $Y_{b'}$ is reduced.  With no loss of generality we will assume that this
holds for every $w\in W_i^0$.  Thus, there is a strong resolution $\wtilde Y \to Y$
of the normalization of $Y$ such that the induced morphism
$\wtilde g: \wtilde Y \to B'$ is semistable \cite{KKMS}.

Following the above construction, there is also a birational map
$X' \dashrightarrow \wtilde Y$.  Let $\mu: \wtilde X\to X'$ denote the birational
morphism obtained by eliminating the indeterminacy of that rational map.  Let
$\sigma: \wtilde X \to \wtilde Y$ be the induced birational morphism, and let
$\wtilde{f'}: \wtilde X\to B'$ denote the resulting family, all fitting in the
following commutative diagram:
\begin{equation}\label{eq:ssDiagram}
  \begin{gathered}
    \xymatrix{
      &  \wtilde X  \ar[dl]_{\mu}  \ar[drr]^{\sigma} \\
      X' \ar@{-->}[rrr] \ar[drrr]_{f'} &&& \wtilde Y \ar[rr] \ar[d]^{\wtilde g} &&
      X  \ar[d]^f  \\
      &&& B' \ar[rr]^{\gamma} && B .  }
  \end{gathered}
\end{equation}
Now, as $X'$ has only canonical singularities 
\cite{KS13},  
for every multiple $m$ of $m_0$, we have
\begin{equation}\label{Refer}
  f'_*\omega^{[m]}_{X'/B'} \simeq  \wtilde{f'}_* \omega^m_{\wtilde X/B'} \simeq \wtilde
  g_* \omega^m_{\wtilde Y/B'}. 
\end{equation}
It follows that
\begin{equation}\label{eq:1stStar}
  \Lm(f') \simeq \Lm(\wtilde g) .
\end{equation}
Moreover, considering the $t^2_m$-fold fiber product of \autoref{eq:ssDiagram} we
find that
\[
\xymatrix{
  &  \wtilde X^{(t^2_m)}  \ar[dl]_{\mu'}  \ar[drr]^{\sigma'} \\
  (X')^{t^2_m} \ar@{-->}[rrr] \ar[drrr]_{(f')^{t^2_m}} &&& \wtilde Y^{(t^2_m)}
  \ar[rr] \ar[d]^{\wtilde g^{(t^2_m)}} &&   X^{(t^2_m)}  \ar[d]^{f^{(t^2_m)}}  \\
  &&& B' \ar[rr]^{\gamma} && B .  }
\]
Again, with $(f')^{t^2_m}: (X')^{t^2_m}\to B'$ being stable, $(X')^{t^2_m}$ has only
canonical singularities and therefore we have
\begin{equation}\label{eq:2ndStar}
  (f')_*^{t^2_m} \omega^{[m]}_{(X')^{t^2_m}/B'}  \simeq (\wtilde{f'})^{t^2_m}
  \omega^m_{\wtilde X^{(t^2_m)}/B'}  
  \simeq \wtilde g_*^{(t^2_m)} \omega^m_{\wtilde Y^{(t^2_m)}/B'}  .
\end{equation}
Combining \autoref{eq:1stStar} and \autoref{eq:2ndStar} now leads to
\[
  \sK_m^{\nu}(f') \simeq \sK_m^{\nu}(\wtilde g) .
\]
On the other hand, according to \autoref{lem:SSBaseChange}, for a log resolution
$(\wtilde X, \wtilde \Delta) \to (X, \Delta)$, we have an injection
\[
  \sK^{\nu}_m(\wtilde g) \hooklongrightarrow \gamma^* \sK_m^{\nu}(\wtilde f) ,
\]
that is an isomorphism over $\gamma^{-1}V$. With
$\cH_{B'}= \gamma^* \cH_B$, this implies that for every multiple $r$ of
$\wtilde \beta^i_m$ we have
\begin{equation}\label{eq:injects}
  \sym^r \sK^{\nu}_m(\wtilde g) \otimes \cH_{B'}^{m\bdot r} \hooklongrightarrow 
  \gamma^*\big(  \sym^r\sK^{\nu}_m(\wtilde f)  \otimes \cH^{m\bdot r}_B   \big)  .
\end{equation}
We saturate the image of the injection \autoref{eq:injects}, call it $\sG'$ and
consider $\bigotimes_{g\in G} g^* \sG'$, where $G= \Gal(B'/B)$, with
$|G| = |\Gal(W'_{B,i} / W_{B,i} )|$. 
By \cite{MR2665168}*{Thm.~4.2.15} and by construction, after removing a subset of $B$
of $\codim_{B}\geq 2$ if necessary, there is a locally free sheaf with an injection
\begin{equation}\label{eq:genisom}
  \jmath: \sG \longinto 
  \bigotimes^{|G|}  \big(   \sym^r\sK^{\nu}_m(\wtilde f) \otimes \mathcal
  H_B^{m\bdot r} \big)   
\end{equation}
on $B$ such that
\begin{enumerate}
\item $\gamma^*\sG\simeq \bigotimes_g g^*\sG'$, and
\item $\jmath$ is an isomorphism over $V$.
\end{enumerate}
Now, as the functor $\gamma_*(\sblank)^G$ is exact, and since 
$\sG'$ is globally generated over $\gamma^{-1}V$, we find that
$\sG\simeq \gamma_* \big( \gamma^*\sG \big)^G$ is generically generated by global sections and
thus so is
\[
  \bigotimes^{|G|} \big( \sym^r\sK^{\nu}_m(\wtilde f) \otimes \cH_B^{m\bdot k}
  \big) \simeq \bigotimes^{|G|} \big( \sym^{r} \sK_m^{\nu}(\wtilde f) \big) \otimes
  \cH_B^{m\bdot k\bdot |G|} .
\]
We conclude the proof by setting $a_i: = |\Gal ( W'_{B,i}/ W_{B,i} )|$.
\end{proof}


\subsection{Generic deformation invariance and an upper-bound for
  $\alpha_{m,\beta_m,s_m}(\sblank)$}

\begin{theorem}
\label{thm:uniform}
In the situation of \autoref{SETUP}, for each $m, \wtilde \beta^i_m$ and $\beta^i_m$
as in \autoref{lem:gg}, for each $i$, there exists an $\alpha^i_{m} \in \bN$ with the
following property.  For every $f_U\subseteq X_{W^0_{V,i}}$ and smooth
compactification $f:X\to B$ there is a section
\[
  0 \neq s_m \in H^0 \left( X^{(t^2_m)} , \big( \omega_{X^{(t^2_m)/B}} \otimes
    (f^{(t^2_m)})^*\big( \Lm^{-1} \otimes \cH_B \big) \big)^{m\bdot \beta^i_m}
  \right) ,
\]
such that
\[
  \alpha_{s_m}(f) \leq \alpha^i_{m} .
\]
\end{theorem}

\begin{proof}
  According to \autoref{lem:gg} there exists an $m_0(n,\nu)\in \bN$ such that, for
  any multiple $m$ of $m_0$, the two sheaves
  \begin{equation}\label{sheaves}
    \begin{gathered}
      \bigotimes^{a_i} \big[ \sym^{\wtilde \beta^i_m} \big(
      (f^{(t^2_m)}_{{W_{B,i}}})_* \omega^m_{X_{W_{B,i}}^{t^2_m}/W_{B,i}} \otimes
      \Lm^{-m}(f_{W_{B,i}}) \big) \big] \otimes \cH^{\beta^i_m\bdot m}, \text{ and}\\
      \bigotimes^{a_i} \big[ \sym^{\wtilde \beta^i_m} \big( f_*^{(t^2_m)}
      \omega^m_{X^{(t^2_m)}/B} \otimes \Lm^{-m}(f) \big) \big] \otimes
      \cH_B^{\beta^i_m\bdot m}
    \end{gathered}
  \end{equation}
  are generically generated by global sections, for some for $\beta^i_m$,
  $\wtilde \beta_m^i$ and $a_i$ as in \autoref{lem:gg}.

  On the other hand, by the invariance of plurigenera \cite{Siu98}
  \hmarginpar{\tiny S: we should give a reference for this. \\
    B: Done. \\ \color{blue} S: Thanks. }%
  and \autoref{item:gg} the two natural maps
  \begin{equation}\label{eq:surject}
    \begin{gathered}
      (f_{{W^0_V},i}^{t^2_m})^* (f_{{W^0_V},i}^{t^2_m})_*
      \omega^m_{X^{t^2_m}_{W^0_V,i}/W_{V,i}} \longrightarrow
      \omega^m_{X^{t^2_m}_{W^0_{V,i}}/ W^0_{V,i}}, \text{ and} \\ 
      (f^{t^2_m})^* f_*^{t^2_m} \omega^m_{U^{t^2_m}/V} \longrightarrow
      \omega^m_{U^{t^2_m}/V}
    \end{gathered}
  \end{equation} 
  are surjective. Therefore, by pulling back the sheaves in \autoref{sheaves},
  and using the surjectivity of the maps in \autoref{eq:surject}, it follows that
  \begin{multline*}
    \bigotimes^{a_i} \sym^{\wtilde \beta^i_m} \Big[ \omega^m_{ X^{(t^2_m)}_{W_{B,i}}
      / W_{B,i} } \otimes (f_{X_{W_B,i}}^{(t^2_m)})^* \Lm^{-m}(f_{{W_{B,i}}}) \Big]
    \otimes (f_{{W_{B,i}}}^{(t^2_m)})^* \cH^{\beta^i_m\bdot m}  \simeq \\
    \simeq \Big[ \omega_{X^{(t^2_m)}_{W_{B,i}}/W_{B,i}} \otimes
    (f_{{W_{B,i}}}^{(t^2_m)})^* \big( \Lm^{-1}(f_{{W_{B,i}}}) \otimes \cH \big)
    \Big]^{\beta^i_m\bdot m},
  \end{multline*}
  and similarly
  \begin{equation}\label{eq:generate}
    \Big[   \omega_{X^{(t^2_m)}/B} \otimes ( f^{(t^2_m)} )^* \big(
    \Lm^{-1}(f)\otimes \cH_B  \big)  \Big]^{\beta^i_m\bdot m} 
  \end{equation}
  are globally generated over the (preimage of the) smooth locus of the family.  
  Next, let
  \begin{equation}\label{BIG}
    0  \neq s_{m, W^0_{V,i}} \in H^0 \left( X_{W^0_{V,i}}^{t^2_m} ,   \Big[
      \omega_{X^{t^2_m}_{W_V,i}/W_{V,i}} \otimes  
      (f_{W^0_{V,i}}^{t^2_m})^* \big( \Lm^{-1}(f_{W^0_{V,i}}) \otimes \cH
      \big)  \Big]^{\beta^i_m\bdot m} \right) 
  \end{equation}
  be such that the intersection of $D_{X_{W^0_{V,i}}} := (s_{m, W^0_{V,i}} )_0$ with
  the general subfamily $f_U \subset X_{W^0_{V,i}}$ is smooth.  Denote the (finite
  type) subscheme of $W^0_{i}$ over which this intersection is not smooth by $T^0_i$.
  Next, for each $f_U$ as above, set
  \begin{equation}\label{SMALL}
    s^0_m  \in H^0 \left(  U^{t^2_m} , \Big[  \omega^{t^2_m}_{U^{t^2_m}/V} 
    \otimes (f_U^{t^2_m})^* \big( \Lm^{-1}(f_U)  \otimes  \cH_B \big)   \Big]  \right)
  \end{equation}
  to be the section induced by the pullback $s_{m, W^0_{V,i}} \big|_{U^{t^2_m}}$.
  Noting that the pullback of the line bundle in \autoref{BIG} is isomorphic to the
  one in \autoref{SMALL}, by the construction of Viehweg numbers \autoref{def:V} and
  \autoref{lem:VNIso}, we find that
  \begin{equation}\label{eq:Star}
    \alpha_{s_{m, W^0_{V, i}}} (f_{{W^0_{V,i}}})  =  \alpha_{s^0_m} (f_U).  
  \end{equation}
  Without loss of generality, we may assume that the sheaf in \autoref{eq:generate}
  is globally generated over $U^{t^2_m}$.  Consequently, there exists an
  \begin{equation}\label{eq:section}
    s_m \in H^0\left( X^{(t^2_m)}, \Big[   \omega_{X^{(t^2_m)}/B} \otimes ( f^{(t^2_m)}
      )^* \big( \Lm^{-1}(f)\otimes \cH_B  \big)  \Big]^{\beta^i_m\bdot m} 
    \right)
  \end{equation}
  such that $s_m|_{U^{t^2_m}} = s^0_m $.  We have that
  $\alpha_{s_m} (f)= \alpha_{s^0_m}(f_U)$ (see
  \autoref{rk:restrict}), so by \autoref{eq:Star}, for every family
  $(f_U: U\to V)\subseteq X_{W^0_{V,i}}\setminus f^{-1}_{W^0_{V,i}}(T_i^0\times V)$,
  regardless of a choice of compactification $f:X\to B$ over $B$, there exists a
  section $s_m$ as in \autoref{eq:section} such that
  \[
    \alpha_{s_m} (f) = \alpha_{s_{m, \beta^i_m, W^0_{V,i}}}
    (f_{{W^0_{V,i}}}).
  \]

  After pulling back $f_{W^0_{V,i}}$ over each irreducible component
  $T^0_{ij}\times V$ of $T^0_i \times V$, and replacing $W_i$ in \autoref{lem:gg} by
  $T^0_{{ij}}$, the conclusions of \autoref{lem:gg} are again valid and we can repeat
  the above argument. As $T_i^0$ is of finite type, we can find an $\alpha^i_m$, by
  induction on the dimension of $W_i$, such that for every
  $f_U\subseteq X_{W^0_{V,i}}$ we have
  \[
    \alpha_{s_m}(f) \leq \alpha_m^i ,
  \]
  with $s_m$ being as in \autoref{eq:section}.
\end{proof}

\begin{corollary}
\label{cor:KEY}
In the situation of \autoref{SETUP} there are integers $m_0$ and $\overline\alpha_m$,
depending only on $n$ and $\nu$, with the following property.  For every multiple $m$
of $m_0$ and every $f_U \in \smn(V)$ and smooth compactification $f: X\to B$ we have
\[
\alpha_{s_m} (f) \leq \overline \alpha_m ,
\]
for some $\beta_m\in \bN$ and $0\neq s_m \in H^0(X^{(t^2_m)}, \sM^{m\bdot \beta_m})$ (as defined in \autoref{def:M}).
\end{corollary}

\begin{proof}
  By \autoref{thm:uniform} we know that for each fixed $i$ and every
  $f_U\in \smn(V)$ with $f_U\subseteq X_{W_{V,i}^0}$ there are
  $\alpha_m^i, \beta^i_m \in \bN$ and $s_m$ such that
\begin{equation}\label{eq:bound}
  \alpha_{s_m} (f) \leq  \alpha_m^i .
\end{equation}

Now, for each irreducible component $T_{ij}$ of the scheme $W_{i}\setminus W_{i}^0$,
we replace $W_{i}$ in \autoref{eq:BigFams} by $T_{ij}$, $1\leq j\leq l$, for some
$l\in \bN$.  Again, by \autoref{lem:gg} and \autoref{thm:uniform} we find that, for
suitable choices of $\beta^i_m$, $s_m$ and $\alpha^i_m$, the inequality
(\ref{eq:bound}) is valid. As $W_i$ is of finite type, by induction on $\dim(W_i)$,
the existence of $\overline \alpha^i_m$ such that
$\alpha_m(f) \leq \overline\alpha^i_m$ holds for every $f_U \subseteq X_{W_V,i}$.  We
conclude the proof by setting
$\overline \alpha_m: = \max_{1\leq i\leq k} \{ \overline\alpha^i_m \}$.
\end{proof}

\section{Higher dimensional Arakelov inequalities}
\label{sect:Section5-APineq}

\subsection{Reflexive systems of Hodge sheaves containing $\Lm(-D)$.} Following
\cite[Def.~2.2]{Taji20}, given sheaves of $\sO_B$-modules $\sW$ and $\sF$, a
\emph{$\sW$-valued system} means a splitting $\sF = \bigoplus \sF_i$ and a sheaf
homomorphism $\tau: \sF \to \sW \otimes \sF$ that is Griffiths-transversal.  If we
assume further that $\sW= \Omega^1_B(\log D)$, $\tau$ is integrable and $\sF$ is
reflexive, then $(\sF, \tau)$ is referred to as a \emph{reflexive logarithmic-system
  of Hodge sheaves}.

\hmarginpar{\tiny S: Should we define ``reflexive systems of Hodge sheaves'', or at
  least give a reference? \\
  B: Done. See the footnote in the next page. \\ \color{blue} S: How about moving
  that footnote her to the beginning as a Definition? \\
  \color{red} B: Done. \\ \color{green} S: Great, thanks!}%

Now, let $V$ be a smooth quasi-projective variety with a smooth compactification
$(B,D)$ and let $\cH_B$ be an ample line bundle on $B$.  Fix $h\in \bZ[x]$. According
to \autoref{thm:uniform} and \autoref{cor:KEY} there are integers $m_0=m_0(n,\nu)$,
$\beta_m$ (suppressing the unnecessary superscript $i$), and $\overline \alpha_m$
such that, for every multiple $m$ of $m_0$, and every smooth compactification
$f:X\to B$ of any smooth projective family $f_U \in \smn(V)$, over an open subset
$B^0\subseteq B$, with $\codim_B(B\backslash B^0) \geq 2$, there exists an
$0 \neq s_m\in H^0( X^{(\tma 2)} , \sM^{m\bdot \beta_m} )$, where (recall from
\autoref{item:SMOOTH})
\begin{equation}\label{eq:TheLine}
  \sM:= \omega_{X^{(\tma 2)/B}} \otimes (f^{(\tma 2)})^*( \Lm^{-1} \otimes \cH_B ) ,
\end{equation}
for which we have $\alpha_{s_m}(f) \leq \overline \alpha_m$.

\begin{proposition}
\label{prop:Hodge}
For every projective morphism $f:X\to B$ as above there is a reflexive system of
Hodge sheaves $\left(\sG= \bigoplus_{i=0}^w \sG_i, \theta\right)$ of weight
$w \in \bN$ on $B$, with logarithmic poles along $D$, satisfying the following properties.
\begin{enumerate}
\item \label{prop:1} $w \leq \dim(X/B) \bdot \tma 2$.
\item  \label{prop:2} $\rank(\sG) \leq \overline\alpha_m$.
\item \label{prop:3} There is an injection
  $\Lm(-D)\otimes \cH_B^{-1} \hooklongrightarrow \sG_0$.
\item \label{prop:4} The torsion free sheaf $\sN_j:= \ker (\theta|_{\sG_j})$ is
  weakly negative, for every $ 0\leq j \leq w$.
\end{enumerate}
\end{proposition}

\hmarginpar{\tiny S: Should we perhaps reference our other paper here? \\
B: Done. \\ \color{blue} S: Thanks.}%

\begin{proof}
  Following \cite{Vie-Zuo03a}, \cite{Taji18}*{2.2} and \cite{Taji20}*{2.2} 
  (see also \cite{KT21}*{\S 4} for a  
  general construction for flat families) we
  consider the $\Omega_{B^0}^1 (\log D)$-valued system
  $(\sF = \bigoplus \sF_i, \tau)$ of weight equal to $\dim(X^{(\tma 2)}/B)$, with
  each $\sF_i$ being defined by
\[
  \myR^i f^{(\tma 2)}_* \left( \Omega^{\tma 2 -i}_{X^{(\tma 2)}/B} (\log
  \Delta_{f^{(\tma 2)}}) \otimes \sM^{-1} \right) .
\]
With $s_m$ as above, over $B_0$, there is a surjective morphism
$Z_{s_m} \to X^{(\tma 2)}$, as in \autoref{eq:Cyclic}. Let $\mathcal V^0$ denote
Deligne's extension \hmarginpar{\tiny S: I think it's either ``the Deligne extension''
  or ``Deligne's extension'', but not ``the Deligne's extension''}%
of the $\bC$-VHS of weight $\dim(Z_{s_m}/B)$ underlying the smooth locus of
$g_{s_m}$, and set $(\sE^0 = \bigoplus \sE_i^0, \theta^0)$ to be the associated Hodge
bundle.  By \cite{Taji18}*{pp.~8--9} there is a morphism of systems
$\Phi: (\sF, \tau) \to (\sE^0, \theta^0)$ such that $\Phi_0$ is injective and its
image $(\sG^0 = \bigoplus_{i=0}^w \sG^0_i, \theta^0)$ has the following property: the
natural extension of $(\sG^0, \theta^0)^{**}$ to the reflexive system of Hodge
sheaves $(\sG, \theta)$ on $B$ satisfies \autoref{prop:4}.  The construction and the
injectivity of $\Phi_0$ implies \autoref{prop:3}.  Moreover, by the construction of
$\sG$ and \autoref{def:V}, we have that $w \leq \dim(X/B) \bdot \tma 2$, so
\autoref{prop:1} follows. Furthermore, $\rank(\sG) \leq \alpha_{m, \beta_m, s_m}(f)$
by construction, so \autoref{prop:2} follows from \autoref{cor:KEY}.
\end{proof}

\subsection{Proof of \autoref{thm:main}}
We are now ready to prove the main theorem.

\begin{theorem}[the precise version of \autoref{thm:main}]\label{MAIN}
  Let $(B,D)$ be a smooth compactification of a smooth quasi-projective variety $V$
  of dimension $d$ and $H$ an ample Cartier divisor on $B$. Further let $n,\nu\in\bN$
  and assume that $K_B+D$ is pseudo-effective.  Then, there exists an
  $m_0= m_0(n,\nu) \in \bN$, such that for every integer multiple $m$ of $m_0$, there
  exist $a_m, b_m, \gamma_m \in \bN$ with the following property. For each smooth
  compactifcation $f: X\to B$ of an arbitrary $f_U \in \smn(V)$ we have
\begin{equation}\label{eq:Arineq}
  c_1(\det f_*\omega^m_{X/B})  \cdot H^{d-1} \leq 
  \left(  d^{\gamma_m}\bdot a_m (K_B  + D) + b_m D \right) \cdot H^{d-1} +
  H^d  . 
\end{equation}
\end{theorem}

\begin{proof}
  Let $\Lm = \det \left( ( f_* \omega^m_{X/B} ) (-m\bdot D) \right)$
  (cf.~\autoref{not:L}). Further let $\sA_m: = \Lm(-D -H)$. As before, set
  $r_m:= \rank(f_* \omega^m_{X/B})$.  We will distinguish cases based on the sign of
  $c_1(\sA_m) \cdot H^{d-1}$.
  
  First, assume that $c_1(\sA_m) \cdot H^{d-1} \leq 0$.
  Then
  \[
    c_1\left( \det (f_*\omega^m_{X/B}) \left( - ( m\bdot r_m) D - H \right) \right)
    \cdot H^{d-1} \leq D \cdot H^{d-1},
  \]
  which implies that
  \begin{equation}\label{Inequality1}
    c_1(\det f_*\omega^m_{X/B}) \cdot H^{d-1}  \leq  (1+ m\bdot r_m)  D \cdot
    H^{d-1}  +  H^{d} . 
  \end{equation}
  
  Next, assume that
  \begin{equation}\label{eq:assump}
    c_1(\sA_m)\cdot H^{d-1} >0.
  \end{equation} 
  According to \autoref{prop:Hodge}, there exists a reflexive system of Hodge sheaves
  $(\sG =\bigoplus_{i=1}^w \sG_i,\theta)$ such that
  $\sA_m \hooklongrightarrow \sG_0$.  Now, define
  \[
    \theta^l : = \underbrace{ (\id \otimes \theta ) \circ \ldots \circ (\id \otimes
      \theta)}_{ \text{ $(l-1)$-times } } \circ \theta : \sG_0 \longrightarrow \left(
    \Omega^1_B (\log D) \right)^{\otimes l} \otimes \sG_l .
  \]
  \begin{claim}\label{claim:DB}
    $\theta(\sA_m) \neq 0$.
  \end{claim}

  \begin{proof}[{Proof of \autoref{claim:DB}}]
  The proof is the same as in \cite{KT21}*{Claim~5.5}.
  \end{proof}

  \hmarginpar{\tiny S: I don't think this follows (that is needs)
    \autoref{claim:DB}. Isn't this true for any $l>w$?  {\bf B: I don't think it can
      be avoided. I am claiming that $\theta = \theta^1$ doesn't kill $\sA_m$.  Of
      course some $\theta^l$ will kill this but the claim is it CANNOT be when
      $l=1$. } \\ S: Uh, I get it. You mean that it follows from the claim that $l$
    has to be larger than $1$. I don't think this is clear from the way you wrote it,
    so I rewrote it. Please check. } %
  Clearly, $\theta^l ( \sA_m ) =0$ for $l\gg 0$ and \autoref{claim:DB} implies that
  such an $l$ has to be larger than $1$. Let $k$ be the largest integer for which
  $\theta^k(\sA_m ) \neq 0$. In particular, then $\theta^{k+1}(\sA_m ) =0$, and hence
  %
  \[
    \theta^k \in \Gamma \left( (\Omega^1_B(\log D))^{\otimes k} \otimes \sHom(\sG_0 ,
      \sN_k) \right),
  \]
  where $\sN_k:= \ker (\theta|_{\sG_k})$ as in \autoref{prop:4}.  As $\sA_m$ is of
  rank one, the nontrivial map
  \[
    \sA_m \longrightarrow \left( \Omega^1_B(\log D) \right)^{\otimes k} \otimes \sN_k
  \]
  must be an injection.  \hmarginpar{\tiny S: Is this what you meant to write here?
    (that it is an injection)? \\
    B: Yes.}%
  \hmarginpar{\vskip-7em\tiny S: why do we need these to
    conclude that this is an injection? \\
    B: We are only using \autoref{claim:DB} to ensure the first time we apply
    $\theta$ to $\sA_m$ the image is not zero. \\ \color{blue} S: I have no idea what
    I had in mind when I wrote this... }%
  We may assume with no loss of generality that the image of this injection is
  saturated.  After raising it to the power $s:= \rank(\sN_k)$, we find that
   \[
     \sA_m^s \otimes (\det \sN_k)^{-1} \hooklongrightarrow \left(\Omega^1_B(\log D)
     \right)^{\otimes k\bdot s} .
   \]
   \hmarginpar{\vskip-4em\tiny S: ``implying'' what? Do we really need anything for
     the injection other than it is not trivial and the left hand side is a line
     bundle? \\
     B: You are right. We don't need anything for this second injection; we are only
     doing the following: assuming that the initial injection is saturated we find
     that it remains saturated after being raised to power $s$. Since the final map
     will be between locally frees we may argue in codimension one. First we dualize
     to get the surjection:
     $ \Big( \Omega^1_B(\log D)^{\otimes ks } \Big)^*\otimes (\det \sN_k)^{-1} \to
     \sA_m^{-s}.$ We get the desired injection by tensoring both sides by
     $\det (\sN_k)$ and the dualizing again.  I am not sure if this needs an
     explanation, although I am embarrassed that my first explanation was certainly
     terrible.  I cleaned up the explanation. Please check. \\ \color{blue} S: OK,
     thanks!}%
   i.e., by using the definition of $\sA_m$, we have
   \begin{equation}\label{eq:Injects}
    \left(  \Lm(-D- H) \right)^{s}  \otimes (\det \sN_k)^{-1}
    \hooklongrightarrow  
    \left( \Omega^1_B(\log D) \right)^{\otimes k\bdot s}  .
  \end{equation}
  
  Now, as $(\det \sN_k)^{-1}$ is pseudo-effective by \autoref{prop:4}, we have the inequality
  \[
    s\bdot c_1 \left( \Lm(-D- H) \right) \cdot H^{d-1} %
    \leq c_1\left( \left( \Lm(-D- H) \right)^s \otimes (\det \sN_k)^{-1} \right)
    \cdot H^{d-1} .
  \]
  On the other hand, by \cite[Thm.~1.3]{CP16} it follows from \autoref{eq:Injects} that
  \[
    c_1\left( \left( \Lm(-D- H) \right)^s \otimes (\det \sN_k)^{-1} \right) \cdot
    H^{d-1} \leq c_1 \left( \left(\Omega^1_B(\log D) \right)^{\otimes k\bdot s}
    \right) \cdot H^{d-1} .
  \]
  By combining these latter two inequalities we find that
  %
  \hmarginpar{\tiny S: I am probably missing something trivial, but where does the
    $-1$ in the exponent of $d$ come from?  \bf{B: This is from the equality :
      $c_1(V^{\otimes t}) = t r^{t-1} c_1(V) $, where $r$ is the rank of $V$.  To get
      this, I use the multiplicavity of $ch(\sblank)$, i.e.
      $ch(V\otimes W) = ch(V) \cdot ch(W)$, to deduce
      $c_1(V\otimes V') = r v_1(V') + c_1(V) r' $.  }} %
  \[
    c_1 \left( \Lm(-D- H) \right) \cdot H^{d-1} %
    \leq d^{k\bdot s -1} \bdot k \bdot (K_B+D) \cdot H^{d-1} .
  \]
  \reversemarginpar Now, substituting the definition
  $\Lm = \det \left( (f_*\omega^m_{X/B}) (-m D) \right)$, this implies that
  \begin{equation}\label{semifinal}
    c_1\left( \det f_*\omega^m_{X/B}  \left(  (-m \bdot r_m -1) D \right)\right) \cdot 
    H^{d-1}  \leq  
    d^{k\bdot s-1} \bdot k \bdot (K_B+D) \cdot H^{d-1} + H^{d}  . 
  \end{equation}
  Furthermore, by \autoref{prop:Hodge} we have that $k \leq w \leq a_m$, where
  $a_m:= \tma 2 \cdot \dim X/B$, and $s\leq \overline \alpha_m$.  Therefore, from
  \autoref{semifinal} we find
  \begin{equation}\label{Inequality2}
    c_1(\det f_*\omega^m_{X/B})  \cdot H^{d-1} \leq  d^{\overline \alpha_m
      \bdot a_m  -1} \bdot  a_m \bdot  
    (K_B+D) \cdot H^{d-1}  + (1+m r_m) \bdot D \cdot H^{d-1}
    + H^{d} . 
  \end{equation}
  The statement now follows from combining \autoref{Inequality1} and
  \autoref{Inequality2}, with $\gamma_m = \overline \alpha_m\bdot a_m -1$,

  \begin{equation}\label{eq:explicit}
    a_m =  (r_m \bdot m^2 \bdot e_m)\bdot\dim X/B     \;\;\; \text{and} \; \;\;
    b_m =   1 + m \bdot r_m.   \qedhere
  \end{equation}
\end{proof}

\subsection{Proof of \autoref{thm:height}}
After removing a subset of $B$ of $\codim_B\geq 2$, let $f': X'\to B'$ be the stable
reduction associated to $\gamma: B'\to B$, as in \autoref{eq:ssDiagram}.  By
\autoref{Refer} and \autoref{VieMor} there is an embedding
$f'_* \omega^{[m]}_{X'/B'} \hooklongrightarrow \gamma^* f_* \omega^m_{X/B}$, such
that \hmarginpar{\tiny{B: Is [KSB] the right reference here? [BV00] are very vague
    about this reference.}}
\[
  \hskip-10em\underbrace{\big[\det f'_* \omega^{[m]}_{X'/B'}
    \big]^{p_m}\hskip-1em}_{\hskip12em\simeq \Phi^*(\lambda_m)\text{
      cf.~\cite{Kollar90}*{2.5},\cite{Viehweg95}*{Thm.~7.17}}}
  \hskip-9em\hooklongrightarrow \gamma^* ( \det f_* \omega^m_{X/B} )^{p_m} .
\]
By Teissier's inequality \cite{Laz04-I}*{Thm.~1.6.1} (and references therein) and
\autoref{thm:main} it follows that
\[
  \left( \vol( \Phi^*\lambda_m) \right)^{\frac{1}{d}} \bdot \Big[ \left( \gamma^*
    H^d \right)^{\frac{1}{d}} \Big]^{d-1} \leq p_m \bdot \deg \gamma \bdot
  \left( \left(d^{\gamma_m} a_m(K_B+D) + b_m D\right) \cdot H^{d-1} +
    H^d \right) ,
\]
which implies that
\[
  \left(\vol( \Phi^*\lambda_m) \right)^{\frac{1}{d}} \leq \frac{p_m \bdot
    \deg\gamma}{ (H^d)^{1-\frac{1}{d}} } \bdot \left( \left(d^{\gamma_m} a_m(K_B+D) +
      b_m D\right) \cdot H^{d-1} + H^d \right).
\]
Finally, setting
\[
  c_m(x_1, x_2, x_3)= \frac{p_m^d}{x_3^{d-1}} \left( d^{\gamma_m} \bdot a_m x_1 + b_m
    x_2 + x_3 \right)^d 
\]
completes the proof. \qed


\begin{bibdiv}
\begin{biblist}

  
\bib{Abramovich-Karu00}{article}{
      author={Abramovich, Dan},
      author={Karu, Kalle},
       title={Weak semistable reduction in characteristic 0},
        date={2000},
        ISSN={0020-9910},
     journal={Invent. Math.},
      volume={139},
      number={2},
       pages={241\ndash 273},
      review={\MR{1738451 (2001f:14021)}},
}


\bib{MR4098246}{article}{
   author={Adiprasito, Karim},
   author={Liu, Gaku},
   author={Temkin, Michael},
   title={Semistable reduction in characteristic 0},
   journal={S\'{e}m. Lothar. Combin.},
   volume={82B},
   date={2020},
   pages={Art. 25, 10},
   review={\MR{4098246}},
}

  
\bib{Arakelov71}{article}{
      author={Arakelov, Sergei~J.},
       title={Families of algebraic curves with fixed degeneracies},
        date={1971},
        ISSN={0373-2436},
     journal={Izv. Akad. Nauk SSSR Ser. Mat.},
      volume={35},
       pages={1269\ndash 1293},
      review={\MR{MR0321933 (48 \#298)}},
}

\bib{BB20}{article}{
      author={Brotbek, Damien},
      author={Brunebarbe, Yohan},
       title={Arakelov-{N}evanlinna inequalities for variations of {H}odge
  structures and applications},
        date={2020},
        note={Preprint
  \href{https://arxiv.org/abs/2007.12957}{arXiv:2007.12957}},
}

\bib{BG71}{article}{
      author={Bloch, Spencer},
      author={Gieseker, David},
       title={The positivity of the Chern classes of an ample vector bundle},
        date={1971},
     journal={Invent. Math.},
      volume={12},
      number={2},
       pages={112\ndash 117},
         url={https://doi.org/10.1007/BF01404655},
}

\bib{BGG06}{article}{
      author={Bradlow, S.B.},
      author={Garc{\'i}a-Prada, O.},
      author={Gothen, P.B.},
       title={Maximal surface group representations in isometry groups of
  classical hermitian symmetric spaces},
        date={2006},
     journal={Geom. Dedicata},
      volume={122},
       pages={185\ndash 213},
}

\bib{MR3143705}{article}{
      author={Bhatt, Bhargav},
      author={Ho, Wei},
      author={Patakfalvi, {\relax Zs}olt},
      author={Schnell, Christian},
       title={Moduli of products of stable varieties},
        date={2013},
        ISSN={0010-437X},
     journal={Compos. Math.},
      volume={149},
      number={12},
       pages={2036\ndash 2070},
         url={http://dx.doi.org/10.1112/S0010437X13007288},
      review={\MR{3143705}},
}

\bib{Bedulev-Viehweg00}{article}{
      author={Bedulev, Egor},
      author={Viehweg, Eckart},
       title={On the {S}hafarevich conjecture for surfaces of general type over
  function fields},
        date={2000},
        ISSN={0020-9910},
     journal={Invent. Math.},
      volume={139},
      number={3},
       pages={603\ndash 615},
      review={\MR{MR1738062 (2001f:14065)}},
}

\bib{Conrad00}{book}{
      author={Conrad, Brian},
       title={Grothendieck duality and base change},
      series={Lecture Notes in Mathematics},
   publisher={Springer-Verlag},
     address={Berlin},
        date={2000},
      volume={1750},
        ISBN={3-540-41134-8},
      review={\MR{1804902}},
}

\bib{CP16}{article}{
      author={Campana, Fr{\'e}d{\'e}ric},
      author={P{\u a}un, Mihai},
       title={Foliations with positive slopes and birational stability of
  orbifold cotangent bundles},
        date={2019},
     journal={Inst. Hautes {\'E}tudes Sci. Publ. Math.},
      volume={129},
      number={1},
       pages={1\ndash 49},
        note={\href{https://arxiv.org/abs/1508.02456}{arXiv:1508.02456}},
}

\bib{Deligne87}{book}{
      author={Deligne, Pierre},
       title={Un th{\'e}or{\`e}me de finitude la monodromie},
   publisher={Progr. Math.},
        date={1987},
        note={Discrete groups in geometry and analysis (New Haven, Conn.,
  1984)},
}

\bib{EV92}{book}{
      author={Esnault, H{\'e}l{\`e}ne},
      author={Viehweg, Eckart},
       title={Lectures on vanishing theorems},
      series={DMV Seminar},
   publisher={Birkh{\"a}user Verlag},
     address={Basel},
        date={1992},
      volume={20},
        ISBN={3-7643-2822-3},
      review={\MR{MR1193913 (94a:14017)}},
}

\bib{Fuj18}{article}{
      author={Fujino, Osamu},
       title={Semipositivity theorems for moduli problems},
        date={2018},
     journal={Ann. Math.},
      volume={187},
      number={187},
       pages={639\ndash 665},
  note={\href{https://doi.org/10.4007/annals.2018.187.3.1}{DOI:10.4007/annals.2018.187.3.1}},
}

\bib{Fuj78}{article}{
      author={Fujita, Takao},
       title={On {K}{\"a}hler fiber spaces over curves},
        date={1978},
     journal={J. Math. Soc. Japan},
      volume={30},
       pages={779\ndash 794},
}

\bib{Ha77}{book}{
      author={Hartshorne, Robin},
       title={Algebraic geometry},
   publisher={Springer-Verlag},
     address={New York},
        date={1977},
        ISBN={0-387-90244-9},
        note={Graduate Texts in Mathematics, No. 52},
      review={\MR{0463157 (57 \#3116)}},
}

\bib{MR2665168}{book}{
      author={Huybrechts, Daniel},
      author={Lehn, Manfred},
       title={The geometry of moduli spaces of sheaves},
     edition={Second},
      series={Cambridge Mathematical Library},
   publisher={Cambridge University Press},
     address={Cambridge},
        date={2010},
        ISBN={978-0-521-13420-0},
         url={http://dx.doi.org/10.1017/CBO9780511711985},
      review={\MR{2665168 (2011e:14017)}},
}

\bib{HX13}{article}{
      author={Hacon, Christopher~D.},
      author={Xu, Chenyang},
       title={Existence of log canonical closures},
        date={2013},
     journal={Invent. Math.},
      volume={192},
      number={1},
       pages={161\ndash 195},
         url={https://doi.org/10.1007/s00222-012-0409-0},
}

\bib{JZ02}{article}{
      author={Jost, JÃŒrgen},
      author={Zuo, Kang},
       title={Hodge bundles over algebraic curves},
        date={2002},
     journal={J. Alg. Geom},
      volume={11},
       pages={535\ndash 546},
}

\bib {Karu00}{article}{
    AUTHOR = {Karu, Kalle},
     TITLE = {Minimal models and boundedness of stable varieties},
   JOURNAL = {J. Algebraic Geom.},
  FJOURNAL = {Journal of Algebraic Geometry},
    VOLUME = {9},
      YEAR = {2000},
    NUMBER = {1},
     PAGES = {93--109},
      ISSN = {1056-3911},
   MRCLASS = {14J10 (14D20 14E30 14J17)},
  MRNUMBER = {1713521 (2001g:14059)},
MRREVIEWER = {Alessio Corti},
}

     

\bib{Kaw81}{article}{
      author={Kawamata, Y.},
       title={{Characterization of Abelian Varieties}},
        date={1981},
     journal={Compositio Math.},
      volume={43},
       pages={253\ndash 276},
}

\bib{KK08}{article}{
      author={Kebekus, Stefan},
      author={Kov{\'a}cs, S{\'a}ndor~J},
       title={Families of canonically polarized varieties over surfaces},
        date={2008},
        ISSN={0020-9910},
     journal={Invent. Math.},
      volume={172},
      number={3},
       pages={657\ndash 682},
  note={\href{http://dx.doi.org/10.1007/s00222-008-0128-8}{DOI:10.1007/s00222-008-0128-8}.
  Preprint \href{http://arxiv.org/abs/0707.2054}{arXiv:0707.2054}},
      review={\MR{2393082}},
}

\bib{KK10}{article}{
      author={Kebekus, Stefan},
      author={Kov{\'a}cs, S{\'a}ndor~J},
       title={The structure of surfaces and threefolds mapping to the moduli
  stack of canonically polarized varieties},
        date={2010},
        ISSN={0012-7094},
     journal={Duke Math. J.},
      volume={155},
      number={1},
       pages={1\ndash 33},
         url={http://dx.doi.org/10.1215/00127094-2010-049},
      review={\MR{2730371 (2011i:14060)}},
}

\bib{MR4156425}{article}{
   author={Koll\'{a}r, J\'{a}nos},
   author={Kov\'{a}cs, S\'{a}ndor J},
   title={Deformations of log canonical and $F$-pure singularities},
   journal={Algebr. Geom.},
   volume={7},
   date={2020},
   number={6},
   pages={758--780},
   issn={2313-1691},
   review={\MR{4156425}},
   doi={10.14231/ag-2020-027},
 }
 
\bib{KKMS}{book}{
      author={Kempf, George},
      author={Knudsen, Finn},
      author={Mumford, David},
      author={Saint-Donat, Bernard},
       title={Toroidal embeddings i},
      series={Lecture Notes in Mathematics},
   publisher={Springer-Verlag Berlin Heidelberg},
        date={1973},
      volume={399},
         url={https://www.springer.com/gp/book/9783540064329},
}

\bib{KoL10}{article}{
      author={Kov{\'a}cs, S{\'a}ndor~J},
      author={Lieblich, Max},
       title={Boundedness of families of canonically polarized manifolds: A
  higher dimensional analogue of shafarevich's conjecture},
        date={2010},
     journal={Ann. Math.},
      volume={172},
      number={3},
       pages={1719\ndash 1748},
  note={\href{https://annals.math.princeton.edu/2010/172-3/p06}{DOI:10.4007/annals.2010.172.1719}},
}

\bib{KM08a}{article}{
      author={Koziarz, Vincent},
      author={Maubon, Julien},
       title={Representations of complex hyperbolic lattices into rank $2$
  classical {L}ie groups of {H}ermitian type},
        date={2008},
     journal={Geom. Dedicata},
      volume={137},
       pages={85\ndash 111},
}

\bib{KM08b}{article}{
      author={Koziarz, Vincent},
      author={Maubon, Julien},
       title={The {T}oledo invariant on smooth varieties of general type},
        date={2010},
     journal={Crelle J. Reine Angew. Math},
      volume={2010},
      number={649},
       pages={207\ndash 230},
  note={\href{https://doi.org/10.1515/crelle.2010.093}{DOI:10.1515/crelle.2010.093}},
}

\bib{MR2629988}{article}{
      author={Koll{\'a}r, J{\'a}nos},
      author={Kov{\'a}cs, S{\'a}ndor~J},
       title={Log canonical singularities are {D}u {B}ois},
        date={2010},
        ISSN={0894-0347},
     journal={J. Amer. Math. Soc.},
      volume={23},
      number={3},
       pages={791\ndash 813},
         url={http://dx.doi.org/10.1090/S0894-0347-10-00663-6},
  note={\href{http://dx.doi.org/10.1090/S0894-0347-10-00663-6}{DOI:10.1090/S0894-0347-10-00663-6}},
      review={\MR{2629988 (2011m:14061)}},
}

\bib{KM98}{book}{
      author={Koll{\'a}r, J{\'a}nos},
      author={Mori, Shigefumi},
       title={Birational geometry of algebraic varieties},
      series={Cambridge Tracts in Mathematics},
   publisher={Cambridge University Press},
     address={Cambridge},
        date={1998},
      volume={134},
        ISBN={0-521-63277-3},
      review={\MR{2000b:14018}},
}

\bib{KMM87}{incollection}{
      author={Kawamata, Yujiro},
      author={Matsuda, Katsumi},
      author={Matsuki, Kenji},
       title={Introduction to the minimal model problem},
        date={1987},
   booktitle={Algebraic geometry, sendai, 1985},
      series={Adv. Stud. Pure Math.},
      volume={10},
   publisher={North-Holland},
     address={Amsterdam},
       pages={283\ndash 360},
      review={\MR{946243 (89e:14015)}},
}

\bib{Kollar10}{article}{
      author={Koll{\'a}r, J{\'a}nos},
       title={Moduli of varieties of general type},
        date={2010},
        note={Preprint
  \href{https://arxiv.org/abs/1008.0621}{arXiv:1008.0621}},
}

\bib{MR4059993}{article}{
      author={Koll\'{a}r, J\'{a}nos},
       title={Log-plurigenera in stable families},
        date={2018},
        ISSN={2096-6075},
     journal={Peking Math. J.},
      volume={1},
      number={1},
       pages={81\ndash 107},
         url={https://doi.org/10.1007/s42543-018-0002-6},
      review={\MR{4059993}},
}

\bib{ModBook}{book}{
    author = {Koll{\'a}r, J{\'a}nos},
    title = {Families of varieties of general type},
    PUBLISHER = {Cambridge University Press},
    note = {with the collaboration of {K.~Altmann} and {S.~Kov{\'a}cs}},
    note = {to appear},
    year = {2022},
}


\bib{Kol86}{article}{
      author={Koll{\'a}r, J{\'a}nos},
       title={Higher direct images of dualizing sheaves. {I}},
        date={1986},
        ISSN={0003-486X},
     journal={Ann. of Math. (2)},
      volume={123},
      number={1},
       pages={11\ndash 42},
         url={http://dx.doi.org/10.2307/1971351},
      review={\MR{825838 (87c:14038)}},
}

\bib{Kollar90}{article}{
      author={Koll{\'a}r, J{\'a}nos},
       title={Projectivity of complete moduli},
        date={1990},
        ISSN={0022-040X},
     journal={J. Differential Geom.},
      volume={32},
      number={1},
       pages={235\ndash 268},
      review={\MR{1064874 (92e:14008)}},
}
\bib {KSB88}{article}{
  AUTHOR = {Koll{\'a}r, J{\'a}nos},
    AUTHOR = {Shepherd-Barron, N. I.},
     TITLE = {Threefolds and deformations of surface singularities},
   JOURNAL = {Invent. Math.},
  FJOURNAL = {Inventiones Mathematicae},
    VOLUME = {91},
      YEAR = {1988},
    NUMBER = {2},
     PAGES = {299--338},
      ISSN = {0020-9910},
     CODEN = {INVMBH},
   MRCLASS = {14J10 (14D20 14J30 32G10 32G13)},
  MRNUMBER = {922803 (88m:14022)},
MRREVIEWER = {Yujiro Kawamata},
}

\bib{Kovacs00a}{article}{
      author={Kov{\'a}cs, S{\'a}ndor~J},
       title={Algebraic hyperbolicity of f{i}ne moduli spaces},
        date={2000},
        ISSN={1056-3911},
     journal={J. Algebraic Geom.},
      volume={9},
      number={1},
       pages={165\ndash 174},
      review={\MR{1713524 (2000i:14017)}},
}

\bib{Kovacs02}{article}{
      author={Kov{\'a}cs, S{\'a}ndor~J},
       title={Logarithmic vanishing theorems and {A}rakelov-{P}arshin
  boundedness for singular varieties},
        date={2002},
        ISSN={0010-437X},
     journal={Compositio Math.},
      volume={131},
      number={3},
       pages={291\ndash 317},
      review={\MR{2003a:14025}},
}

\bib{Kov13}{book}{
      author={Kov{\'a}cs, S{\'a}ndor~J},
       title={Singularities of stable varieties},
      series={Handbook of moduli. Adv. Lect. Math. (ALM), 25},
   publisher={Int. Press, Somerville, MA},
        date={2013},
      volume={II},
         url={http://dx.doi.org/10.1007/978-94-017-0717-6},
}

\bib{Kovacs96e}{article}{
      author={Kov{\'a}cs, S{\'a}ndor~J},
       title={Smooth families over rational and elliptic curves},
        date={1996},
        ISSN={1056-3911},
     journal={J. Algebraic Geom.},
      volume={5},
      number={2},
       pages={369\ndash 385},
        note={Erratum: {J}.\ {A}lgebraic {G}eom.\ {\bf 6} (1997), no.\ 2, 391},
      review={\MR{1374712 (97c:14035)}},
}

\bib{Kovacs97a}{article}{
      author={Kov{\'a}cs, S{\'a}ndor~J},
       title={On the minimal number of singular f{i}bres in a family of
  surfaces of general type},
        date={1997},
        ISSN={0075-4102},
     journal={J. Reine Angew. Math.},
      volume={487},
       pages={171\ndash 177},
      review={\MR{1454264 (98h:14038)}},
}

\bib{KP17}{article}{
      author={Kov{\'a}cs, S{\'a}ndor~J},
      author={Patakfalvi, {\relax Zs}olt},
     TITLE = {Projectivity of the moduli space of stable log-varieties and
              subadditivity of log-{K}odaira dimension},
   JOURNAL = {J. Amer. Math. Soc.},
  FJOURNAL = {Journal of the American Mathematical Society},
    VOLUME = {30},
      YEAR = {2017},
    NUMBER = {4},
     PAGES = {959--1021},
      ISSN = {0894-0347},
   MRCLASS = {14J10},
  MRNUMBER = {3671934},
       DOI = {10.1090/jams/871},
       URL = {http://dx.doi.org/10.1090/jams/871},
}

\bib{KS13}{article}{
      author={Kov{\'a}cs, S{\'a}ndor~J},
      author={Schwede, Karl},
       title={Inversion of adjunction for rational and {D}u~{B}ois pairs},
        date={2016},
     journal={Algebra Number Theory},
      volume={10},
      number={5},
       pages={969\ndash 1000},
         url={http://dx.doi.org/10.2140/ant.2016.10.969},
}

\bib{KT21}{article}{
      author={Kov{\'a}cs, S{\'a}ndor~J},
      author={Taji, Behrouz},
       title={Hodge sheaves underlying flat projective families},
        date={2021},
        note={Preprint
  \href{https://arxiv.org/abs/2103.03515}{arXiv:2103.03515}},
}

\bib{Laz04-I}{book}{
      author={Lazarsfeld, Robert},
       title={Positivity in algebraic geometry. {I}},
      series={Ergebnisse der Mathematik und ihrer Grenzgebiete. 3. Folge. A
  Series of Modern Surveys in Mathematics [Results in Mathematics and Related
  Areas. 3rd Series. A Series of Modern Surveys in Mathematics]},
   publisher={Springer-Verlag},
     address={Berlin},
        date={2004},
      volume={48},
        ISBN={3-540-22533-1},
        note={Classical setting: line bundles and linear series},
      review={\MR{2095471 (2005k:14001a)}},
}

\bib{Mat72}{article}{
      author={Matsusaka, T.},
       title={Polarized varieties with given {H}illbert polynomial},
        date={1972},
     journal={Amer. J. Math.},
      volume={9},
       pages={1027\ndash 1077},
}

\bib{Parshin68}{article}{
    AUTHOR = {Parshin, A. N.},
     TITLE = {Algebraic curves over function fields. {I}},
   JOURNAL = {Izv. Akad. Nauk SSSR Ser. Mat.},
  FJOURNAL = {Izvestiya Akademii Nauk SSSR. Seriya Matematicheskaya},
    VOLUME = {32},
      YEAR = {1968},
     PAGES = {1191--1219},
      ISSN = {0373-2436},
   MRCLASS = {14.20},
  MRNUMBER = {0257086 (41 \#1740)},
MRREVIEWER = {J.-E. Bertin},
}



\bib{MR2871152}{article}{
      author={Patakfalvi, {\relax Zs}olt},
       title={Viehweg's hyperbolicity conjecture is true over compact bases},
        date={2012},
        ISSN={0001-8708},
     journal={Adv. Math.},
      volume={229},
      number={3},
       pages={1640\ndash 1642},
         url={http://dx.doi.org/10.1016/j.aim.2011.12.013},
      review={\MR{2871152 (2012m:14072)}},
}

\bib{Pet00}{article}{
      author={Peters, Chris A.~M.},
       title={Arakelov-type inequalities for hodge bundles},
        date={2000},
        note={Preprint
  \href{https://arxiv.org/pdf/math/0007102.pdf}{arXiv:math/0007102}},
}

\bib{Shaf63}{incollection}{
    AUTHOR = {Shafarevich, I. R.},
     TITLE = {Algebraic number fields},
 BOOKTITLE = {Proc. Internat. Congr. Mathematicians (Stockholm, 1962)},
     PAGES = {163--176},
 PUBLISHER = {Inst. Mittag-Leffler},
   ADDRESS = {Djursholm},
      YEAR = {1963},
   MRCLASS = {10.65 (12.50)},
  MRNUMBER = {0202709 (34 \#2569)},
      NOTE = {English translation: Amer.\ Math.\ Soc.\ Transl.\ (2) {\bf
      31} (1963), 25--39},
MRREVIEWER = {P. Roquette},
}

\bib{Siu98}{article}{
      author={Siu, Yum-Tong},
       title={Invariance of plurigenera},
        date={1998},
        ISSN={0020-9910},
     journal={Invent. Math.},
      volume={134},
      number={3},
       pages={661\ndash 673},
      review={\MR{1660941 (99i:32035)}},
}

\bib{Taji20}{article}{
      author={Taji, Behrouz},
       title={Birational geometry of smooth families of varieties admitting
  good minimal models},
        date={2020},
        note={Preprint
  \href{https://arxiv.org/abs/2005.01025}{arXiv:2005.01025}},
}

\bib{Taji18}{article}{
      author={Taji, Behrouz},
       title={On the Kodaira dimension of base spaces of families of
  manifolds},
        date={2021},
     journal={Journal Pure Applied Algebra},
      volume={225},
       pages={106729},
  url={https://www.sciencedirect.com/science/article/pii/S0022404921000694},
        note={\href{https://arxiv.org/abs/1809.05616}{arXiv:1809.05616}},
}

\bib{Viehweg83}{incollection}{
      author={Viehweg, Eckart},
       title={Weak positivity and the additivity of the {K}odaira dimension for
  certain fibre spaces},
        date={1983},
   booktitle={Algebraic varieties and analytic varieties (tokyo, 1981)},
      series={Adv. Stud. Pure Math.},
      volume={1},
   publisher={North-Holland},
     address={Amsterdam},
       pages={329\ndash 353},
      review={\MR{715656 (85b:14041)}},
}

\bib{Viehweg95}{book}{
      author={Viehweg, E.},
       title={Quasi-projective moduli for polarized manifolds},
      series={Ergebnisse der Mathematik und ihrer Grenzgebiete (3)},
   publisher={Springer-Verlag},
     address={Berlin},
        date={1995},
      volume={30},
        ISBN={3-540-59255-5},
      review={\MR{1368632 (97j:14001)}},
}
\bib{Viehweg01}{incollection}{
    AUTHOR = {Viehweg, Eckart},
     TITLE = {Positivity of direct image sheaves and applications to
              families of higher dimensional manifolds},
 BOOKTITLE = {School on Vanishing Theorems and Effective Results in
              Algebraic Geometry (Trieste, 2000)},
    SERIES = {ICTP Lect. Notes},
    VOLUME = {6},
     PAGES = {249--284},
 PUBLISHER = {Abdus Salam Int. Cent. Theoret. Phys., Trieste},
      YEAR = {2001},
   MRCLASS = {14H10 (14J10)},
  MRNUMBER = {1919460 (2003f:14024)},
MRREVIEWER = {Arvid Siqveland},
}

\bib{Vie08}{article}{
      author={Viehweg, Eckart},
       title={Arakelov inequalities},
        date={2008},
     journal={Surveys Diff. Geom.},
      volume={13},
       pages={245\ndash 276},
        note={preprint
  \href{https://arxiv.org/abs/0812.3350}{arXiv:0812.3350}.},
}


\bib{Viehweg10}{article}{
      author={Viehweg, Eckart},
       title={Compactifications of smooth families and of moduli spaces of
  polarized manifolds},
        date={2010},
     journal={Ann. of Math.},
      number={172},
       pages={809\ndash 910},
        note={\href{http://arxiv.org/abs/math/0605093}{arXiv:math/0605093}},
}

\bib{VZ02}{incollection}{
      author={Viehweg, Eckart},
      author={Zuo, K.},
       title={Base spaces of non-isotrivial families of smooth minimal models},
        date={2002},
   booktitle={Complex geometry (G{\"o}ttingen, 2000)},
   publisher={Springer},
     address={Berlin},
       pages={279\ndash 328},
      review={\MR{1922109 (2003h:14019)}},
}

\bib{Vie-Zuo03a}{article}{
      author={Viehweg, Eckart},
      author={Zuo, Kang},
       title={On the {B}rody hyperbolicity of moduli spaces for canonically
  polarized manifolds},
        date={2003},
        ISSN={0012-7094},
     journal={Duke Math. J.},
      volume={118},
      number={1},
       pages={103\ndash 150},
      review={\MR{1978884 (2004h:14042)}},
}

\bib{VZ07}{article}{
      author={Viehweg, Eckart},
      author={Zuo, Kang},
       title={Arakelov inequalities and the uniformization of certain rigid
  {S}himura varieties},
        date={2007},
        ISSN={1056-3911},
     journal={J. Diff. Geom.},
      volume={77},
       pages={291\ndash 352},
}

\bib{Zuc82}{article}{
      author={Zucker, Steven},
       title={Remarks on a theorem of {F}ujita},
        date={1982},
     journal={J. Math. Soc. Japan},
      volume={34},
      number={1},
       pages={47\ndash 54},
}


\end{biblist}
\end{bibdiv}

\end{document}
